\documentclass[reqno]{amsart}
\usepackage[utf8]{inputenc}

\usepackage{comment}
\usepackage{bbm}
\usepackage[foot]{amsaddr}
\usepackage{comment}
\usepackage{mathabx}

\usepackage{geometry}
\usepackage{graphicx}
\graphicspath{{Figures/}}
\usepackage[usenames,dvipsnames]{xcolor}
\usepackage{hyperref}
\usepackage{dsfont,url}
\usepackage{subcaption}

\newtheorem{theorem}{Theorem}

\newtheorem{proposition}[theorem]{Proposition}
\newtheorem{definition}[theorem]{Definition}
\newtheorem{corollary}[theorem]{Corollary}
\newtheorem{lemma}[theorem]{Lemma}

\newtheorem{remark}{Remark}

\def\qed{\hbox{${\vcenter{\vbox{		 %HOLLOW SQUARE
   \hrule height 0.4pt\hbox{\vrule width 0.5pt height 6pt
   \kern5pt\vrule width 0.5pt}\hrule height 0.4pt}}}$}}

\def\cB{\mathcal B}

\def\cD{\mathcal D}

\def\cF{\mathcal F}

\def\cH{\mathcal H}

\def\cM{\mathcal M}

\def\cR{\mathcal R}

\def\bE{\mathbb E}

\def\bP{\mathbb P}

\def\eps{\varepsilon}

\newcommand{\ste}[1]{\textcolor{black}{#1}}

\begin{document}

\title{Threshold dynamics of SAIRS epidemic model with Semi-Markov switching}
\author{Stefania Ottaviano}
\address{Stefania Ottaviano \newline University of Trento, Dept.~of Civil, Environmental and Mechanical Engineering, Via Mesiano 77, Trento, 38123, Italy}

\email{stefania.ottaviano@unitn.it}


\maketitle

\begin{abstract}
We study the threshold dynamics of a stochastic SAIRS-type model with vaccination, where the role of asymptomatic and symptomatic infectious individuals is explicitly considered in the epidemic dynamics. In the model, the values of the disease transmission rate may switch between different levels under the effect of a semi-Markov process. We provide sufficient conditions ensuring the almost surely epidemic extinction and persistence in time mean. In the case of disease persistence, we investigate the omega-limit set of the system and give sufficient conditions for the existence and uniqueness of an invariant probability measure.
\end{abstract}

\section{Introduction}
%Mathematical  models  have  become  important  tools  to  analyse  the  spread  and  control of  some  illness. 
Starting  with  the  research  of  Kermack  and  McKendrick \cite{kermack}, in  the last century a huge amount of mathematical epidemic models have been formulated, analysed and applied to a variety  of infectious  diseases, specially during the recent Covid-19 pandemic.

Once an infectious disease developed, the main goal is containing its spread. Several control strategies may be applied, such as detection and isolation of infectious individuals, lockdowns or vaccination. However, the detection of infectious individuals is far from being easy since they may not show symptoms. The presence of asymptomatic cases allows a wide circulation of a virus in the population, since they often remain unidentified and presumably have more contacts that symptomatic cases.  
The contribution of the so called ``silent spreaders'' to the infection transmission dynamics are relevant for various communicable diseases, such as Covid-19, influenza, cholera and shigella \cite{kemper1978effects,stilianakis1998emergence,nelson2009cholera,robinson2013model,ansumali2020modelling, peirlinck2020visualizing,ottaviano2021global};
hence, asymptomatic cases should be considered in such mathematical epidemic models.

The containment of the disease with multiple lockdowns or isolation processes affects the transmission of the disease through the population. %Moreover, some epidemics may vary their behaviour due to climate or environmental changes \cite{grassfras2006,dowel2001}. 
Moreover, in
real biological systems, some parameters of the model are usually influenced by random switching of the
external environment regime. For example, disease transmission rate in some epidemic model is influenced
by random meteorological factors linked to the survival of many bacteria and viruses \cite{serra2013european,small1986relationship}. 
Thus, the transmission rate as the ability of an infectious individual to transmit infection and also as expression of the contact rate between individuals can be subject to random fluctuations.
%{can also be influenced by epidemic containment policies (e.g., introduction of masks and social distancing measures)}. 
Hence, the choice of fixed deterministic parameters in models is unlikely to be realistic. 
In epidemiology, many authors have considered random switching systems (also called hybrid systems), whose distinctive feature is
the coexistence of continuous dynamics and discrete events (random jumps at points in time).
In particular, many works consider regime switching of external environments following a homogeneous continuous-time Markov chain \cite{gray2012sis,greenhalgh2016modelling,li2017threshold,han2013stochastic,wang2019dynamical, ottaviano2020stochastic}. The Markov property facilitates the mathematical analysis, although it can be a limitation as
the sojourn time in each environment is exponentially distributed, which yields constant transition
rates between different regimes. However, in reality, the transition rates are usually time-varying, hence, in each environmental state the conditional holding time distribution can be not exponential. For example, as shown in \cite{serra2013european,small1986relationship} and reported in \cite{li2019threshold}, the dry spell (consisting of consecutive
days with daily rain amount below some given threshold) length distribution is better modeled by Pearson
type III distribution, gamma distribution or Weibull distribution. \\
In this work, in order to include random influences on transmission parameters and overcome the drawback of the Markov setting, we use a semi-Markov process for describing environmental random changes. Semi-Markov switching systems are an emerging topic from both theoretical and practical viewpoints, able of capturing inherent uncertainty and randomness in the environment in many applied fields, ranging
from epidemiology to DNA analysis, financial engineering, and wireless communications \cite{zong2021advances}. Compared to the most common Markov switching systems, they
better characterize a broader range of phenomena but brings more difficulties to their stability
analysis and control. Recently, a semi-Markov switching model has been used to analyze the coexistence and competitiveness of species in ecosystems \cite{li2021coexistence}.
In epidemiology, to the best of our knowledge, there
are only very few semi-Markov switching models \cite{li2019threshold,zhao2020threshold,cao2021basic} and no one of these considers the role of the asymptomatic individuals in the disease dynamics. Thus, in this paper, we want to fill this gap and improve our understanding of these types of hybrid systems.\\

Precisely, we study a SAIRS-type model with vaccination, where the total population $N$ is partitioned into four compartments, namely $S$, $A$, $I$, $R$, which represent the fraction of Susceptible, Asymptomatic infected, symptomatic Infected and Recovered individuals, respectively, such that $N=S+A+I+R$. The infection can be transmitted to a susceptible through a contact with either an asymptomatic infectious individual, at rate $\beta_A$, or a symptomatic individual, at rate $\beta_I$. 
Once infected, all
susceptible individuals enter an asymptomatic state, indicating a
delay between infection and symptom onset if they occur. Indeed, we include in the asymptomatic class both individuals who will never develop the symptoms and pre-symptomatic who will eventually become symptomatic. From the asymptomatic compartment, an
individual can either progress to the class of symptomatic infectious $I$, at rate $\alpha$,
or recover without ever developing symptoms, at rate $\delta_A$. An infected individuals with symptoms can recover at a rate $\delta_I$.
We assume that the recovered individuals do not obtain a long-life immunity and can return to the susceptible state  after an average time $1/\gamma$. We also assume that a proportion $\nu$ of susceptible individuals receive a dose of vaccine which grants them a temporary immunity.
We do not add a compartment for the vaccinated individuals,
not distinguishing the vaccine-induced immunity from the natural one acquired after recovery from the virus.
We consider the vital dynamics of the entire population and, for simplicity, we assume that the rate of births and deaths are the same, equal to $\mu$; we do not distinguish between natural deaths and disease related deaths \cite{ottaviano2021global}.\\
Moreover, we assume that the environmental regimes (or states) influence the infectious transmission rates $\beta_A$ and $\beta_I$, and that may switch under the action of a semi-Markov process. Accordingly, the values of $\beta_A$ and $\beta_I$ switch between different levels depending on the state in which the process is.\\
The paper is organized as follows. In Section \ref{model_des}, we provide some basic concepts of semi-Markov processes and determine the SAIRS model under study. We show the existence of a unique global positive solution, and find a positive
invariant set for the system. In Section \ref{thr_dynamics}, we investigate the threshold dynamics of the model. Precisely, we first consider the case in which $\beta_A(r)=\beta_I(r):=\beta(r)$, $\delta_A(r)=\delta_I(r):=\delta(r)$ and find the basic reproduction number $\cR_0$ for our stochastic epidemic model driven by the semi-Markov process. We show that $\cR_0$ is a threshold value, meaning that its position with respect to one determines the almost surely disease extinction or the persistence in time mean.
Then, we investigate the case $\beta_A(r)\neq \beta_I(r)$ or $\delta_A(r) \neq \delta_I(r)$. First, we find two different sufficient conditions for the almost surely extinction, that are interchangeable, meaning that it is sufficient that one of the two are verified to ensure the extinction. After, we find a sufficient condition for the almost surely persistence in time mean of the system. Thus, we have two not adjacent regions depending on the model parameters, one where the system goes to extinction almost surely, and the other where it is persistent.
In Section \ref{sec:limit_set}, as well as in Section \ref{sec:ipm}, for simplicity, we restrict the analysis to the case of a semi-Markov process with two states. Under the disease persistence condition, we investigate the omega-limit set of the system.
The introduction of the backward recurrence time process, that keeps track of the time elapsed since the latest switch, allows the considered stochastic system to be a piecewise deterministic Markov process \cite{li2019threshold}. Thus, in Section \ref{sec:ipm}, we prove the existence of a unique invariant probability measure by utilizing an exclusion principle in \cite{{stettner1986existence}} and the property of positive Harris recurrence. \\
Finally, in Section \ref{num_exp}, we validate our analytical results via numerical simulations and show the relevant role of the mean sojourn time in each environmental regime in the extinction or persistence of the disease.

\section{Model description and basic concepts}\label{model_des}

\subsection{The Semi-Markov process.}

Let $\{r(t), t \ge 0\}$ be a semi-Markov process
taking values in the state space $\cM=\{1, \ldots, M\}$, whose elements denote the states of external environments influencing the transmission rates value of the model. Let
 $$0=\tau_0 < \tau_1 < \ldots <\tau_n < \ldots$$
be the jump times, and
$$\sigma_1=\tau_1 -\tau_0, \; \sigma_2 = \tau_2-\tau_1, \; \ldots, \; \sigma_n=\tau_n -\tau_{n-1}, \ldots$$
be the time intervals between two consecutive jumps.
Let $(p_{i,j})_{m \times m}$ denote the transition probability matrix
and $F_i(t)$, $t \in [0,\infty)$, $i=1, \ldots, M$, the conditional holding time distribution of the semi-Markov process, then
$$\bP(r(\tau_{n+1})=j, \sigma_{n+1} \leq t| r(\tau_n)=i)=p_{i,j}F_i(t),$$
and the embedded chain $\{ X_n := r(\tau_n), n=0,1, \ldots\}$ of $\{r(t), t \ge 0\}$ is Markov with one-step transition probability $(p_{i,j})_{m \times m}$. Moreover, let $f_i(t)$ represents the density function of the conditional holding time distribution $F_i(t)$, then for any $t \in [0, \infty)$ and $i,j \in \cM$, we define
\begin{align*}
q_{i,j}(t) &:= p_{i,j} \frac{f_i(t)}{1-F_i(t)} \ge 0 \qquad \forall i \neq j, \\
q_{i,i}(t) &:= -\sum_{j \in \cM, j \neq i} q_{i,j}(t) \qquad \forall i \in \cM.
\end{align*}

We give the following same assumptions
 as in \cite{li2019threshold}, which will be valid throughout the paper.\\
 
Assumptions \textbf{(H1)}:

\begin{itemize}
\item[(i)] The transition matrix $(p_{i,j})_{m \times m}$ is irreducible with $p_{i,i}=0$, $i \in \cM$;
\item[(ii)] For each $i \in \cM$, $F_i(\cdot)$ has a continuous and bounded density $f_i(\cdot)$, and $f_i(t)>0$ for all $t \in (0, \infty)$;
\item[(iii)] For each $i \in \cM$, there exists a constant $\eps_i > 0$ such that
$$\frac{f_i(t)}{1-F_i(t)} \ge \eps_i$$
for all $t \in [0,\infty)$. 
\end{itemize}

In \cite{li2019threshold}, the authors provide a list of some probability distributions satisfying the assumption \textbf{(H1)}, and show that the constraint conditions of \textbf{(H1)} are very weak. Specifically, they provide the phase-type distribution (PH-distribution) of a nonnegative random variable, and prove that this PH-distribution, or at most an approximation of it, satisfies the conditions in \textbf{(H1)}. Thus, they conclude that essentially the conditional holding time distribution of the semi-Markov process can be any distribution on $[0,\infty)$.\\

\begin{remark}
 Let us note that in the case of exponential (memoryless) sojourn time distribution, the semi-Markov process $\{r(t), t \ge 0\}$ degenerates into a continuous time Markov chain. That is, if $F_i(t)=1-e^{-q_i t}$ for some $q_i >0$, $i \in \cM$, then
 $$\frac{f_i(t)}{1-F_i(t)} \ge q_i$$
 for all $t \in [0, \infty)$, from which
 \begin{equation*}
q_{i,j}:=q_{i,j}(t)= 
\begin{cases}
q_i p_{i,j} &\text{if } i \neq j\\
-q_i &\text {if } i=j
\end{cases}
\end{equation*}
where $q_{i,j}$ is the transition rates from state $i$ to state $j$, and $q_{i,j}\geq 0$ if $i \neq j$, while $-q_i=q_{i,i}= -\sum_{i \neq j} q_{i,j}$.
Thus, the matrix $Q=(q_{i, j})_{m \times m}$ generates the Markov chain $\{r(t), t \ge 0\}$, i.e.,
\begin{equation*}
\bP\{r(t+ \Delta t)= j| r(t)=i\}=\begin{cases}
q_{i,j}\Delta t + o(\Delta t), \qquad \text{if} \quad i \neq j,\\
1+ q_{i,j}\Delta t + o(\Delta t), \qquad \text{if} \quad i=j,
\end{cases}
\end{equation*}
where $\Delta t >0$ represents a small time increment.

By the assumptions \textbf{(H1)} follows that the matrix $Q$ is irreducible. Under this condition, the Markov chain has a unique stationary positive probability distribution $\pi=(\pi_1, \ldots, \pi_m)^T$ which can be determined by solving the following linear equation $\pi^T Q =0$, subject to $\sum_{r=1}^M \pi_r=1$, and $\pi_r >0$, $\forall r \in \cM$.
\end{remark}

Let us introduce the process
$$\eta(t)=t-\sup\{u <t: r(u) \neq r(t)\},$$
which represents the amount of time the process $\{ r(t), t \ge 0\}$ is at the current state after the last jump. It is also denoted as the backward recurrence time process. The pair $\{(\eta(t),r(t)), t \ge 0\}$ satisfies the Markov property \cite{limnios2001semi}, moreover it is strong Markov \cite[Chapter 6]{hou2013markov}.

\subsection{Model description.} Let us consider a SAIRS model with vaccination, as in \cite{ottaviano2021global}. 
\begin{equation}\label{sairs}
\begin{split}
     \frac{d S(t)}{dt} &= \mu  - \bigg(\beta_A A(t) + \beta_I I(t)\bigg)S(t) -(\mu + \nu) S(t) +\gamma R(t),\\ 
     \frac{d A(t)}{dt} &=  \bigg(\beta_A A(t) + \beta_I I(t)\bigg)S(t) -(\alpha + \delta_A +\mu) A(t), \\ 
     \frac{d I(t)}{dt} &= \alpha A(t) - (\delta_I + \mu)I(t), \\ 
     \frac{d R(t)}{dt} &=  \delta_A A(t) +\delta_I I(t) + \nu S(t) - (\gamma + \mu)R(t).
     \end{split}
\end{equation}

Let us now incorporate the impact of the external random environments into system \eqref{sairs}.
We assume that the external random environment is described by a semi-Markov process.
%{can also be influenced by epidemic containment policies (e.g., introduction of masks and social distancing measures)}.
We only consider the environmental influence on the disease transmission rate
since it may be more sensitive to environmental fluctuations than other parameters of model \eqref{sairs}. 
Thus, the average value of the transmission rate may switch between different levels with the switching of the environmental regimes.
As a result, the SAIRS deterministic model \eqref{sairs} evolves in a random dynamical system with semi-Markov switching of the form
\begin{equation}\label{sairs_s}
\begin{split}
     \frac{d S(t)}{dt} &= \mu  - \bigg(\beta_A(r(t))A(t) + \beta_I(r(t))I(t)\bigg)S(t) -(\mu + \nu) S(t) +\gamma R(t),\\ 
     \frac{d A(t)}{dt} &=  \bigg(\beta_A(r(t))A(t) + \beta_I(r(t))I(t)\bigg)S(t) -(\alpha + \delta_A +\mu) A(t), \\ 
     \frac{d I(t)}{dt} &= \alpha A(t) - (\delta_I + \mu)I(t), \\ 
     \frac{d R(t)}{dt} &=  \delta_A A(t) +\delta_I I(t) + \nu S(t) - (\gamma + \mu)R(t),   %, \qquad i=1, 2, \ldots, n. 
     \end{split}
\end{equation}

Let us introduce $\bar \beta:=(\beta_A,\beta_I)$. If the initial conditions of the driving process $\{(\eta(t),r(t)), t \ge 0\}$ are $\eta(0)=0$ and $r(0)=r_0$, then system \eqref{sairs_s} starts from the initial condition $(S(0), A(0), I(0), R(0))$ and follows \eqref{sairs} with $\bar \beta=\bar \beta(r_0)$ until the first jump time $\tau_1$, with conditional holding distribution $F_{r_0}(\cdot)$. Then, the environmental regime switches instantaneously from state $r_0$ to state $r_1$; thus, the process restarts from the state $r_1$ and the system evolves accordingly to \eqref{sairs} with $\bar \beta=\bar \beta(r_1)$ and distribution $F_{r_1}(\cdot)$ until the next jump time $\tau_2$. The system will evolve in the similar way as long as the semi-Markov process jumps.
This yields a continuous and piecewise
smooth trajectory in $\mathbb R^4$. Let us note that the solution process $\{x(t)=(S(t),A(t),I(t),R(t)), t \ge 0\}$
that records the position of the switching trajectory of \eqref{sairs_s} is not Markov. However, by means of additional components, $\{(x(t), \eta(t), r(t)), t \geq 0 \}$ is a homogeneous Markov process.

In this paper, unless otherwise specified, let $(\Omega,\mathcal{F}, \{\mathcal{F}_t\}_{ t\geq 0}, \bP)$ be a complete
probability space with a filtration $\{\mathcal{F}_t\}_{ t\geq 0}$ satisfying the usual
conditions (i.e. it is right continuous and $\cF_0$ contains all $\bP$-null
sets).\\

Since $S+A+I+R=1,$ system \eqref{sairs_s} is equivalent to the following three-dimensional dynamical system:
\begin{equation}\label{sairs3_s}
\begin{split}
    \frac{d S(t)}{dt} &= \mu -\bigg(\beta_A(r(t))A(t) + \beta_I(r(t))I(t)\bigg)S(t) -(\mu + \nu +\gamma) S(t) + \gamma(1-A(t)-I(t)), \\ 
      \frac{d A(t)}{dt} &=\bigg(\beta_A(r(t))A(t) + \beta_I(r(t))I(t)\bigg)S(t) -(\alpha + \delta_A +\mu) A(t),  \\
     \frac{d I(t)}{dt} &= \alpha A(t) - (\delta_I + \mu)I(t), 
\end{split}
\end{equation}
with initial condition $(S(0), A(0), I(0))$ belonging to the set 
$$\Gamma=\{ (S, A, I) \in \mathbb R_+^{3}| S+ A+ I \leq 1\},$$
where $\mathbb R_+^{3}$ is the non-negative orthant of $\mathbb R^{3}$, and initial state $r(0) \in \cM$.\\

System \eqref{sairs3_s} can be written in matrix notation as
\begin{equation}\label{dx}
    \frac{dx(t)}{dt}=g(x(t),r(t)),
\end{equation}
where $x(t) = (S(t), A(t), I(t))$ and $g(x(t)) = (g_1(x(t)), g_2(x(t)), g_3(x(t)))$ is defined according to \eqref{sairs3_s}.\\
In the following, for any initial value $x(0)$, we denote  by $x(t,\omega,x(0))=( S(t,\omega,x(0)), A(t,\omega,x(0)), I(t,\omega,x(0)))$, the solution of \eqref{dx} at time $t$ starting in $x(0)$, or by $x(t)$ if there is no ambiguity, for the sake of simplicity, and by {$x_r(t)$ the solution of the subsystem $r$}.\\

%\subsection{Global solution and extinction}

In the following theorem we ensure the no explosion of the solution in any finite time, by proving a
somehow stronger property, that is $\Gamma$ is a positive invariant domain for \eqref{sairs3_s}.

\begin{theorem}\label{invset}
  For any initial value $(x(0),\eta(0),r(0)) \in \Gamma \times \mathbb R^+ \times \cM$, and for any choice of system parameters $\beta_A(\cdot)$, $\beta_I(\cdot)$, there exists a unique solution $x(t,\omega,x(0))$ to system \eqref{sairs3_s} on $t \geq 0$. Moreover, for every $\omega \in \Omega$ the solution remains in $\Gamma$ for all $t >0$.
 \end{theorem}

 \begin{proof}
Let $0=\tau_0 < \tau_1 < \tau_2 < \ldots, < \tau_n < \ldots$ be the jump times of the semi-Markov chain $r(t)$, and
let $r(0)=r_0 \in \cM$ be the starting state. Thus, $r(t)=r_0$ on $[\tau_0,\tau_1)$. The subsystem for $t \in [\tau_0, \tau_1)$ has the following form:

\begin{equation*}
\frac{dx(t)}{dt}=g(x(t),r_0),
\end{equation*}
and, for \cite[Thm 1]{ottaviano2021global}, its solution $x(t) \in \Gamma$, for $t \in [\tau_0, \tau_1)$ and, by continuity for $t = \tau_1$, as well. Thus, $x(\tau_1) \in \Gamma$ and by considering $r(\tau_1)=r_1$, the subsystem for $t \in [\tau_1, \tau_2)$
becomes

\begin{equation*}
\frac{dx(t)}{dt}=g(x(t),r_1).
\end{equation*}
Again, $x(t) \in \Gamma$, on $t \in [\tau_1, \tau_2)$ and, by continuity for $t = \tau_2$, as well. Repeating this process continuously, we obtain the claim. %that the solution $x(t)$ of system \eqref{dx} remains in $\Gamma$ with probability one, $for all $t \geq 0$. 
\end{proof}

 As the switching concerns only the infection rates $\beta_A$ and $\beta_I$, all the subsystems of \eqref{sairs3_s} share the same disease-free equilibrium (DFE) 

$$x_0= \left(S_0, A_0, I_0\right)=\left(\frac{\mu+\gamma}{\mu+\nu+\gamma}, 0,0\right).$$

Now, we report results related to the stability analysis of each deterministic subsystems of \eqref{sairs3_s} corresponding to the state $r$, $r=1, \ldots, M$. The proof of the following results can be found in \cite{ottaviano2021global}, where the global stability of the deterministic model \eqref{sairs} is investigated.

\begin{lemma}\label{propR0}
The basic reproduction number $\cR_0$ of the subsystem of \eqref{sairs3_s} corresponding to the state $r$ is given by
\begin{equation}\label{R0}
\cR_0(r)  = \left ( \beta_A(r) + \dfrac{\alpha \beta_I(r)}{\delta_I + \mu}\right) \dfrac{\gamma + \mu}{(\alpha+\delta_A + \mu)(\nu + \gamma + \mu)}, \qquad r=1, \ldots, M.
\end{equation}
\end{lemma}

Now, let us define
\begin{equation}\label{FminusV1}
 (F-V)(r) = \left( 
\begin{matrix}
\beta_A(r) S_0 -(\alpha+\delta_A+\mu) & \beta_I(r) S_0  \\ \alpha & -(\delta_I+\mu)
\end{matrix}\right), 
\end{equation}
where $F$ and $V$ are the matrices in equations $(6)$ and $(7)$ defined in \cite{ottaviano2021global}.

\begin{lemma}\label{spFV}
Let us fix $r \in \cM$. The matrix $(F-V)(r)$ related to the subsystem $r$ of \eqref{sairs3_s} has a real spectrum. Moreover, if $\rho(FV^{-1}(r)) <1$, all the eigenvalues of $(F-V)(r)$ are negative. 
\end{lemma}

\begin{theorem}\label{glob_attr}
Let us fix $r \in \cM$. The disease-free equilibrium $x_0$ is globally asymptotically stable for the subsystem $r$ of \eqref{sairs3_s} if $\cR_0(r)<1$.
\end{theorem}

\begin{lemma}\label{ex_ee}
Let us fix $r\in \cM$. The endemic equilibrium $x^*_r = (S^*(r),A^*(r),I^*(r))$ exists and it is unique in $\mathring{\Gamma}$ for the subsystem $r$ of \eqref{sairs3_s} if $\cR_0(r)>1$. Moreover, $x^*_r$ is locally asymptotically stable in $\mathring{\Gamma}$.
\end{lemma}

\begin{theorem}\label{thm:globeq}
Let  us  fix $r\in \cM$ and assume that $\beta_A(r) = \beta_I(r) =: \beta(r)$ and $\delta_A = \delta_I =: \delta$. The endemic equilibrium $x^*_r = (S^*(r),A^*(r),I^*(r))$ is globally asymptotically stable {in $\mathring{\Gamma}$} for the subsystem $r$ of \eqref{sairs3_s} if $\cR_0(r)>1$.
\end{theorem}

\begin{theorem}\label{thm:globdiff}
 Let us fix $r\in \cM$, and consider $\beta_A(r) \neq \beta_I(r)$ or $\delta_A \neq \delta_I$. Assume that $\mathcal{R}(r)_0 > 1$ and $\beta_A(r) < \delta_I$. Then, the endemic equilibrium $x^*$ is globally asymptotically stable in $\mathring{\Gamma}$ for the subsystem $r$ of\eqref{sairs_s}.
\end{theorem}

\section{Threshold dynamics of the model}\label{thr_dynamics}
By the assumptions \textbf{(H1)} the embedded Markov chain $\{ X_n, n \in \mathbb N\}$, associated to the semi-Markov process $\{ r(t), t \geq 0\}$ has a unique stationary positive probability distribution $\pi=(\pi_1, \ldots, \pi_M)$. Let
\begin{equation*}
m_i=\int_0^\infty [1-F_i(u)] du
\end{equation*}
be the mean sojourn time of $\{ r(t), t \geq 0\}$ in state $i$. Then, by the Ergodic theorem \cite[Thm 2, p. 244]{gikhman2004theory}, we have that for any bounded measurable function $f:(E,\mathcal E) \to (\mathbb R_+, \mathcal B(\mathbb R^+)),$
\begin{equation}\label{ergsemi}
\lim_{t \to \infty} \frac{1}{t} \int_0^t f(r(s)) ds= \frac{\sum_{r \in \cM}f(r) \pi_r m_r}{\sum_{r \in \cM} \pi_r m_r} \qquad \text{a.s.}
\end{equation}

Hereafter, we denote $$\widecheck{\beta} := \max_{r \in M} \{\beta_A(r),\beta_I(r)\}.$$

\subsection{$\beta_A(r)=\beta_I(r):=\beta(r)$, $r=1, \ldots, M$, $\delta_A=\delta_I:=\delta$.\\}

\begin{theorem}\label{ext_switch1}
Let us assume $\beta_A(r)=\beta_I(r):=\beta(r)$, $\delta_A=\delta_I:=\delta$ in each subsystem $r=1, \ldots, M$. If 
\begin{equation*}\label{condext1}
\sum_{r \in \cM} \pi_r m_r \left( \beta(r) \frac{\gamma+\mu}{\nu +\gamma + \mu } -(\delta+\mu)\right) <0,
\end{equation*}

%\ste{(cambia lettera $m_r$?)}

then the solution of system \eqref{sairs3_s} with any initial value $(x(0),\eta(0),r(0)) \in \Gamma \times \mathbb R^+ \times \cM$ satisfies

\begin{align}
    \lim_{t \to +\infty} S(t) &= \frac{\gamma+\mu}{\nu +\gamma + \mu }=:S_0 \qquad \text{a.s.}, \label{cond_1}\\
    \lim_{t \to +\infty} A(t) &=0 \qquad \text{a.s.}, \label{cond_2}\\
    \lim_{t \to +\infty} I(t) &=0 \qquad \text{a.s}. \label{cond_3}
    \end{align}
\end{theorem}

\begin{proof}

We know that for all $\omega \in \Omega$, it holds

\begin{equation*}
\frac{dS(\omega,t)}{dt} \leq \mu+ \gamma  -(\mu+\nu+\gamma) S(\omega,t).   
\end{equation*}

 For different selections of sample point $\omega \in \Omega$, the sample path $S(\omega,t)$ may have different convergence speeds with respect to time $t$. Thus, for any $\omega \in \Omega$ and any constant $\eps >0$, by the comparison theorem, there exists $T(\omega,\eps)>0$, such that for all $t > T$
 
$$S(\omega,t) \leq S_0 + \eps,$$
hence
 \begin{equation}\label{limS2}
\limsup_{t \to \infty} S(t) \leq S_0, \qquad \text{a.s.}
\end{equation}
Based on this consideration, we shall prove assertions \eqref{cond_2} and \eqref{cond_3}. We have that for all $\omega \in \Omega,$ and $t > T$

\begin{align*}
\frac{d \ln(I(t)+A(t))}{dt}& = \beta(r(t))S(t)-(\delta+\mu) \\
&\leq \beta(r(t))(S_0+\eps) -(\delta+\mu)  .
\end{align*}
This implies that 

\begin{equation*}
\ln(I(t)+A(t)) \leq \ln(I(T)+A(T)) + \int_T^t (\beta(r(u))(S_0+\eps)-(\delta+\mu))\; du 
\end{equation*}
from which, by the ergodic result for semi-Markov process \eqref{ergsemi}, we get

\begin{align*}
    \limsup_{t \to \infty} \frac{\ln(I(t)+A(t))}{t}& \leq \limsup_{t \to \infty }\frac{1}{t} \int_T^t (\beta(r(u))(S_0+\eps)-(\delta+\mu))\; du \\
    %&= \sum_{r \in \cM} \pi_r (\beta(r)(S_0+\eps)-(\delta+\mu)) \qquad \text{a.s.} (Markov case)
    &= \frac{1}{\sum_{r \in \cM} \pi_r m_r}\bigg[ \sum_{r \in \cM} \pi_r m_r (\beta(r)(S_0+\eps)-(\delta+\mu)) \bigg] \qquad \text{a.s.}
\end{align*}

If %$\sum_{r \in \cM} \pi_r \left( \beta(r) \frac{\gamma+\mu}{\nu +\gamma + \mu } -(\delta+\mu)\right)  <0$
$ \sum_{r \in \cM} \pi_r m_r(\beta(r)(S_0+\eps)-(\delta+\mu))<0 $, then for sufficiently small $\eps >0$, we have $\sum_{r \in \cM} \pi_r (\beta(r)(S_0+\eps)-(\delta+\mu)) <0$, and consequently 

 \begin{equation}\label{limAI}
     \lim_{t \to +\infty} A(t) =0, \qquad \text{and} \qquad \lim_{t \to +\infty} I(t) =0 \qquad \text{a.s}.
 \end{equation}
 
 Now, we shall prove assertion \eqref{cond_1}. Let $\bar \Omega=\{ \omega \in \Omega : \lim_{t \to +\infty } A(t)=0\} \cap \{ \omega \in \Omega : \lim_{t \to +\infty } I(t)=0\}$. Then, from \eqref{limAI}, $\bP(\bar \Omega)=1$.
Then, for any $\omega \in \bar \Omega$ and any constant $\eps >0$, there exists $T_1(\omega,\eps)>0$, such that for all $t > T_1$
$$A(\omega,t) < \eps, \qquad I(\omega,t) < \eps.$$ Thus,  we have for all $\omega \in \bar \Omega$, and $t > T_1$
 \begin{align*}
\frac{d S(\omega,t)}{dt} &\ge  \mu - \eps \beta(\ste{r}) S(\omega,t) -(\mu + \nu +\gamma) S(\omega,t) + \gamma(1-2\eps)\\
& \ge  \mu - \eps \widecheck \beta S(\omega,t) -(\mu + \nu +\gamma) S(\omega,t) + \gamma(1-2\eps).
\end{align*}
Following the same arguments in the proof of Theorem \ref{glob_attr}, we can assert that

$$\liminf_{t\to\infty}S(\omega,t)\ge S_0, \qquad \omega \in \bar \Omega.$$
Recalling that $P(\bar \Omega)=1$, we have
$$\liminf_{t\to\infty}S(t)\ge S_0 \qquad \text{a.s.}$$
that combined with \eqref{limS2} gives us that
$$\lim_{t \to \infty} S(t)=S_0 \qquad \text{a.s.}$$ 
\end{proof}

Thus, under the condition of Theorem \ref{ext_switch1}, we can say that any positive solution of system \eqref{sairs3_s} converges exponentially to the disease-free state $x_0=(S_0,0,0)$ almost surely.

Based on the definition in \cite{bacaer2013basic,li2019threshold}, the \emph{basic reproduction number}, for our model \eqref{sairs3_s} with $\beta_A(r)=\beta_I(r):=\beta(r)$, $\delta_A=\delta_I:=\delta$, in the semi-Markov random environment, can be written from Theorem \ref{ext_switch1} as
%number by which all infection rates should be divided to lead the epidemic to the critical situation where neither exponential growth nor exponential decay occurs.
\begin{equation}\label{r0semi}
\cR_0=\frac{\sum_{r \in \cM} \pi_r m_r \beta(r) S_0}{\sum_{r \in \cM}\pi_r m_r (\delta+\mu)}.
\end{equation}
We notice that we would have arrived to the same result if we had followed the same arguments as in \cite{li2019threshold}, that are based on the theory of basic reproduction in random environment in \cite{bacaer2013basic}.

\begin{remark}
It easy to see that in the case of Markov-switching, that is for the exponential holding time distribution in each regime, the basic reproduction number for our model \eqref{sairs3_s} with $\beta_A(r)=\beta_I(r):=\beta$, $\delta_A=\delta_I:=\delta$ is
 \begin{equation*}\label{R0equal}
\cR_0=\frac{\sum_{r \in \cM} \pi_r  \beta(r) S_0}{\sum_{r \in \cM}\pi_r (\delta+\mu)}=\frac{\sum_{r \in \cM} \pi_r  \beta(r) S_0}{ (\delta+\mu)}.
\end{equation*}
\end{remark}

\begin{proposition}\label{prop:ROsemi}
From \eqref{r0semi}, the following alternative conditions are valid
\begin{itemize}
    \item [(i)] $\cR_0 <1$ if and only if $\sum_{r \in \cM} \pi_r m_r (\beta(r) S_0-(\delta+\mu)) <0$,
    \item [(ii)] $\cR_0 >1$ if and only if $\sum_{r \in \cM} \pi_r m_r (\beta(r) S_0-(\delta+\mu)) >0$.
\end{itemize}
\end{proposition}
The proof is immediate, so it is omitted.\\

\subsection{$\beta_A(r)\neq \beta_I(r)$, or $r=1, \ldots, M$, $\delta_A\neq\delta_I$}

Let us define

\begin{equation*}
    \widecheck \beta(r)=\max \{ \beta_A(r), \beta_I(r)\},  \qquad \hat \delta=\min \{ \delta_A, \delta_I\}.
\end{equation*}
\begin{theorem}\label{thm:extdiff1}
Let $\beta_A(r)\neq \beta_I(r)$, or $\delta_A\neq\delta_I$ in each subsystem $r=1,\ldots, M$. If
\begin{equation}\label{condextdiff1}
  \sum_{r \in \cM} \pi_r m_r (\widecheck \beta(r)S_0-(\hat \delta+\mu))<0,  
\end{equation}
then the solution of system \eqref{sairs3_s} with any initial value $(x(0),\eta(0),r(0)) \in \Gamma \times \mathbb R^+ \times \cM$ satisfies

\begin{align}
    \lim_{t \to +\infty} S(t) &= \frac{\gamma+\mu}{\nu +\gamma + \mu }=:S_0 \qquad \text{a.s.}, \label{cond_21}\\
    \lim_{t \to +\infty} A(t) &=0 \qquad \text{a.s.}, \label{cond_22}\\
    \lim_{t \to +\infty} I(t) &=0 \qquad \text{a.s}. \label{cond_23}
    \end{align}
\end{theorem}

\begin{proof}
Let us prove conditions \eqref{cond_22} and \eqref{cond_23}. By using equation \eqref{limS2}, we have that for all $\omega \in \Omega,$ and $t > T$

\begin{align*}
\frac{d \ln(I(t)+A(t))}{dt}& \leq \widecheck \beta(r(t))(S_0+\eps)-(\hat \delta-\mu). \\
\end{align*}
By the same arguments as in Theorem \ref{ext_switch1}, we obtain that

\begin{align*}
    \limsup_{t \to \infty} \frac{\ln(I(t)+A(t))}{t}& \leq 
    %&= \sum_{r \in \cM} \pi_r (\beta(r)(S_0+\eps)-(\delta+\mu)) \qquad \text{a.s.} (Markov case)
     \frac{1}{\sum_{r \in \cM} \pi_r m_r}\bigg[ \sum_{r \in \cM} \pi_r m_r (\widecheck \beta(r)(S_0+\eps)-(\hat \delta+\mu)) \bigg] \qquad \text{a.s.}
\end{align*}

Thus, if %$\sum_{r \in \cM} \pi_r \left( \widecheck \beta(r) \frac{\gamma+\mu}{\nu +\gamma + \mu } -(\delta+\mu)\right)  <0$
$ \sum_{r \in \cM} \pi_r m_r(\widecheck \beta(r)S_0-(\delta+\mu))<0 $, then for sufficiently small $\eps >0$, we have $\sum_{r \in \cM} \pi_r (\widecheck \beta(r)(S_0+\eps)-(\delta+\mu)) <0$, and consequently 

 \begin{equation*}
     \lim_{t \to +\infty} A(t) =0, \qquad \text{and} \qquad \lim_{t \to +\infty} I(t) =0 \qquad \text{a.s}.
 \end{equation*}

To prove assertion \eqref{cond_21}, we follow the same steps as in Theorem \eqref{ext_switch1}, by considering that $\beta_A(r)$ and $\beta_I(r)$ are less than or equal to $\widecheck\beta$.
\end{proof}

With a different proof we can find another sufficient condition for the extinction of the disease.
\begin{theorem}\label{thm:extdiff2}
Let $\beta_A(r)\neq \beta_I(r)$ or $\delta_A\neq\delta_I$ in each subsystem $r=1,\ldots, M$, and let $B(r)=(F-V)(r)$ as in \eqref{FminusV1}. If
\begin{equation}\label{condext2}
%\frac{1}{\sum_{r \in \cM}\pi_r m_r} 
\sum_{r \in \cM}  \pi_r m_r \lambda_1(B(r)+B(r)^T)<0,
\end{equation}
 where $\lambda_1$ is the maximum eigenvalue, then the solution of system \eqref{sairs3_s} with any initial value $(x(0),\eta(0),r(0)) \in \Gamma \times \mathbb R^+ \times \cM$ satisfies
\begin{align}
    \lim_{t \to +\infty} S(t) &= \frac{\gamma+\mu}{\nu +\gamma + \mu } \qquad \text{a.s.}, \label{cond1_1}\\
    \lim_{t \to +\infty} A(t) &=0 \qquad \text{a.s.}, \label{cond1_2}\\
    \lim_{t \to +\infty} I(t) &=0 \qquad \text{a.s.} \label{cond1_3}
    \end{align}
\end{theorem}

\begin{proof}
By following the same arguments as in the proof of Theorem \ref{ext_switch1}, we know that \eqref{limS2} holds.
 Thus, we have that for any $\omega \in \Omega$ and any constant $\eps >0$, there exists $T(\omega,\eps)>0$, such that for all $t > T$ 
\begin{align*}
    \frac{d A(t)}{dt} & \leq \bigg(\beta_A(r(t))A(t) + \beta_I(r(t))I(t)\bigg)(S_0 + \eps) -(\alpha + \delta_A +\mu) A(t),  \\
     \frac{d I(t)}{dt} &= \alpha A(t) - (\delta_I + \mu)I(t). \nonumber 
    \end{align*}
We shall now prove assertions \eqref{cond1_2} and \eqref{cond1_3}, by considering the comparison system
\begin{align*}\label{dw1}
    \frac{d w_1(t)}{dt} & = \bigg(\beta_A(r(t)) w_1(t) + \beta_I(r(t)) w_2(t)\bigg)(S_0 + \eps) -(\alpha + \delta_A +\mu) w_1(t),  \\
     \frac{d w_2(t)}{dt} &= \alpha w_1(t) - (\delta_I + \mu)w_2(t), \qquad w_1( \bar T)=A( T),\quad w_2(  T)=I( \bar T) \nonumber 
    \end{align*}
    Let $w(t)=(w_1(t),w_2(t))^T$ and consider the function $V(w(t))= \ln ||w(t)||_2  $. Then, let $B_{\eps}(r)=(F_{\eps}-V_{\eps})(r)$ the matrix in \eqref{FminusV1}, computed in $x_0(\eps)=(S_0+\eps, 0,0)$. Then, we have
    
    \begin{align*}
    \frac{dV(w(t))}{dt}&= \frac{1}{||w(t)||^2_2} \langle w(t), \dot w(t)\rangle = \frac{w(t)^T}{||w(t)||_2} B_\eps(r)\frac{w(t)}{||w(t)||_2}  \\
    &= \frac{w(t)^T}{||w(t)||_2}   \frac{\big((B_\eps)(r) + (B_\eps)^T(r) \big)}{2} \frac{w(t)}{||w(t)||_2}\\ &\leq \frac{\lambda_1\big((B_\eps)(r) + (B_\eps)^T(r) \big)}{2}.   
    \end{align*}
    By the same arguments in Theorem \ref{ext_switch1}, invoking \eqref{ergsemi}, assertions \eqref{cond1_2} and \eqref{cond1_3} follows, and consequently \eqref{cond1_1}.

\end{proof}

\begin{remark}
 In the case of Markov-switching, the condition \eqref{condext2} becomes
 $$\sum_{r \in \cM}  \pi_r \lambda_1(B(r)+B(r)^T)<0.$$
\end{remark}

\begin{remark}
 Let us fix $r \in \cM$. Let us consider condition \eqref{condextdiff1} and $\cR_0(r)$ in \eqref{propR0}. It is easy to see that it holds
 
$$ \widecheck \beta(r)S_0-(\hat \delta+\mu)<0 \Rightarrow \cR_0(r) < 1$$

 Now, let us consider condition \eqref{condext2}. We have that
 
 $$\lambda_1(B(r)+B(r)^T)=\beta_A(r)S_0-(\alpha+\delta_A+\mu)-(\delta_I+\mu)+ \sqrt{(\beta_A(r)S_0 -(\alpha+\delta_A)+ \delta_I)^2+(\beta_I(r)S_0 + \alpha)^2}.$$
  From this and from \eqref{R0}, it is easy to see that
  $$\lambda_1(B(r)+B(r)^T)<0 \Rightarrow \cR_0(r) < 1,$$ 
  and that if $\beta_I(r)S_0=\alpha$, it holds
  $$\lambda_1(B(r)+B(r)^T)<0 \Leftrightarrow \cR_0(r) < 1.$$ \\\\
\end{remark}
In Section \ref{num_exp}, we will compare numerically the two conditions \eqref{condextdiff1} and \eqref{condext2}, by showing a case in which condition \eqref{condextdiff1} is satisfied but \eqref{condext2} does not, and the other case in which the vice versa occurs. Thus, it is sufficient that one of the two conditions is verified to ensure the almost sure extinction.

\section{Persistence}
In this section, we investigate the persistence in time mean of the disease.
\begin{definition}
We say that system \eqref{sairs3_s} is \emph{almost surely persistent in the time mean}, if
\begin{equation*}
    \liminf_{t \to \infty} \frac{1}{t} \int_0^t S(u) du > 0 \qquad \liminf_{t \to \infty} \frac{1}{t} \int_0^t (I(u)+A(u)) du > 0,
\end{equation*}
with probability one.
\end{definition}

Let us remark that $I+A$ denote the fraction of individuals that may infect the susceptible population.

\begin{theorem}\label{pers1}
 Let us assume $\beta_A(r)=\beta_I(r):=\beta(r)$, $\delta_A=\delta_I:=\delta$ in each subsystem $r=1, \ldots, M$. If $\cR_0>1$, then for any initial value $(x(0),\eta(0),r(0)) \in \mathring \Gamma \times \mathbb R^+ \times \cM$, the following statement is valid with probability 1:
 \begin{equation}\label{persS}
   \liminf_{t \to \infty} \frac{1}{t} \int_0^t S(u) du \geq \frac{\mu}{\mu+\nu+\widecheck \beta}, 
 \end{equation}
 
 \begin{equation}\label{persAI}
     \liminf_{t \to \infty} \frac{1}{t} \int_0^t (I(u)+A(u)) du \geq \frac{\mu+\nu+\gamma}{\widecheck \beta(\widecheck \beta+\gamma)}\frac{\sum_{r \in \cM} \pi_r m_r (\beta(r)S_0-(\delta+\mu))}{\sum_{r \in \cM} \pi_r m_r}.
 \end{equation}
 \end{theorem}
 
\begin{proof}
For ease of notation, we will omit the dependence on $\omega$ and that on $t$ if not necessary. Since $A+I \leq 1$, from the first equation of system \eqref{sairs_s}, we have
\begin{equation*}
 \frac{dS(t)}{dt} \geq \mu-(\mu+\nu) S - \beta(r)(A+I) S \geq \mu-[(\mu+\nu)+\beta(r)]S,
\end{equation*}
integrating the above inequality and dividing both sides by $t$, we obtain 
\begin{equation}\label{intt}
   \frac{1}{t}(\mu+\nu) \int_0^t S(u) du + \frac{1}{t}\int_0^t \beta(r(u))S(u) du \geq \mu - \frac{S(t)-S(0)}{t}.
\end{equation}
 Then, for all $\omega \in \Omega$ it holds
 \begin{equation}\label{stso}
   \lim_{t \to +\infty}\frac{S(t)-S(0)}{t}=0, \qquad \text{a.s.}   
 \end{equation}
From \eqref{intt}, it follows
$$\liminf_{t \to +\infty} \int_0^t S(u) ds \geq \frac{\mu}{\mu+\nu+\widecheck \beta}, \qquad \text{a.s.}$$
and assertion \eqref{persS} is proved. \\
Next, we will prove assertion \eqref{persAI}. By summing the second and third equation of \eqref{sairs3_s}, we have that

\begin{equation*}
\frac{d \ln(I(t)+A(t))}{dt}= \beta(r)S- \delta - \mu ,    
\end{equation*}
from which, integrating both sides
\begin{equation}\label{lnIA}
\begin{split}
\ln (I(t)+A(t))&=\ln(I(0)+A(0)) + \int_0^t \beta(r(u)) S(u) du -\int_0^t (\delta+\mu) du\\ 
&= \ln(I(0)+A(0)) +  \int_0^t \beta(r(u))\frac{\gamma+\mu}{\nu+\gamma+ \mu} du - \int_0^t \beta(r(u))\left(\frac{\gamma+\mu}{\nu+\gamma+ \mu} -S(u)\right) du -\int_0^t (\delta+\mu) du \\
& \geq \ln(I(0)+A(0)) +  \int_0^t \beta(r(u))\frac{\gamma+\mu}{\nu+\gamma+ \mu} du -\widecheck \beta \int_0^t \left(\frac{\gamma+\mu}{\nu+\gamma+ \mu} -S(u)\right) du -\int_0^t (\delta+\mu) du.
\end{split}
\end{equation}
Now, we have that
\begin{equation}\label{ds2}
\begin{split}
\frac{d S(t)}{dt} &= \mu-(\mu+\nu)S -\beta(r) (A+I)S +\gamma -\gamma S -\gamma (I+A) \\
& \geq(\nu+ \gamma +\mu) \left( \frac{\gamma + \mu}{\nu+ \gamma +\mu}-S\right)-(\beta(r)+\gamma)(I+A)\\
&\geq (\nu+ \gamma +\mu) \left( \frac{\gamma + \mu}{\nu+ \gamma +\mu}-S\right)-(\widecheck \beta+\gamma)(I+A),
\end{split}    
\end{equation}
from which, integrating both sides,
\begin{equation}\label{sintdis}
\begin{split}
(\nu+ \gamma +\mu) \int_0^t \left( \frac{\gamma + \mu}{\nu+ \gamma +\mu}-S(u)\right) du \leq S(t)-S(0) + (\widecheck \beta+\gamma) \int_0^t (I(u)+A(u)) du.
\end{split}    
\end{equation}
Combining \eqref{lnIA} with \eqref{sintdis}, we obtain
\begin{equation}\label{lnIAdis}
\begin{split}
  \ln (I(t)+A(t))& \geq \ln(I(0)+A(0)) + \int_0^t \left(\beta(r(u))\frac{\gamma+\mu}{\nu+\gamma+ \mu} -(\delta +\mu)\right) du\\
  & -\frac{\widecheck \beta}{\nu+\gamma+ \mu}\left[ S(t)-S(0) +(\widecheck \beta +\gamma) \int_0^t (I(u)+A(u)) du\right].
  \end{split}
\end{equation}
For all $\omega \in \Omega$, \eqref{stso} holds, moreover it is easy to see that
$$\limsup_{t \to +\infty} \frac{\ln (I(t)+A(t))}{t} \leq 0,$$
thus from \eqref{lnIAdis}, and the ergodic result \eqref{ergsemi}, we get
$$\liminf_{t \to +\infty} \frac{1}{t} \int_0^t (I(u)+A(u)) du \geq \frac{\nu+\gamma+ \mu}{\widecheck \beta(\widecheck \beta+\gamma)}\frac{\sum_{r \in \cM} \pi_r m_r (\beta(r)S_0-(\delta+\mu))}{\sum_{r \in \cM} \pi_r m_r}, \qquad \text{a.s.}$$
that is assertion \eqref{persAI}.
\end{proof} 
 
Thus, by (ii) of Proposition \ref{prop:ROsemi}, we conclude that the disease is persistent in the time mean with probability 1. \\

\begin{corollary}\label{cor:persI}
Let us assume $\beta_A(r)=\beta_I(r):=\beta(r)$, $\delta_A=\delta_I:=\delta$ in each subsystem $r=1, \ldots, M$. If $\cR_0>1$, then for any initial value $(x(0),\eta(0),r(0)) \in \mathring \Gamma \times \mathbb R^+ \times \cM$, the following statements hold with probability 1:
\begin{equation}\label{persI}
     \liminf_{t \to \infty} \frac{1}{t} \int_0^t I(u) du \geq \frac{\alpha}{\alpha+\delta+\mu}\frac{\mu+\nu+\gamma}{\widecheck \beta(\widecheck \beta+\gamma)}\frac{\sum_{r \in \cM} \pi_r m_r (\beta(r)S_0-(\delta+\mu))}{\sum_{r \in \cM} \pi_r m_r},
 \end{equation}
 and
 \begin{equation}\label{persA}
     \liminf_{t \to \infty} \frac{1}{t} \int_0^t A(u) du \geq \frac{\delta+\mu}{\alpha+\delta+\mu}\frac{\mu+\nu+\gamma}{\widecheck \beta(\widecheck \beta+\gamma)}\frac{\sum_{r \in \cM} \pi_r m_r (\beta(r)S_0-(\delta+\mu))}{\sum_{r \in \cM} \pi_r m_r}.
 \end{equation}
\end{corollary} 
\begin{proof}
Integrating the third equation of system \eqref{sairs3_s}
 and dividing both sides by $t$, we have
 $$(\alpha + \delta +\mu) \frac{1}{t} \int_0^t I(u) du \geq \frac{\alpha}{t} \int_0^t (I(u)+A(u)) -\frac{I(t)-I(0)}{t}, $$
 Thus, from \eqref{persAI} it holds \eqref{persI}. %add \alpha \int_0^t I(u)du to both sides of the eq
 
 Now, as before, by integrating the third equation of system \eqref{sairs3_s}
 and dividing both sides by $t$, it is easy to see the \eqref{persA} holds. %do not add \int_0^t I(u)du to both sides of th eq
\end{proof}

For the next result, we need to define 

\begin{equation*}
    \hat \beta(r)=\min \{ \beta_A(r), \beta_I(r)\}, \qquad \widecheck{ \hat \beta}= \max_{r \in \cM} \hat\beta(r), \qquad \text{and} \qquad \widecheck \delta=\max \{ \delta_A, \delta_I\}.
\end{equation*}
\begin{theorem}\label{pers2}
 Let us assume $\beta_A(r)\neq \beta_I(r)$ or $\delta_A \neq \delta_I$ in each subsystem $r=1, \ldots, M$. If
 \begin{equation}\label{condpersdiff}
     \sum_{r \in \cM} \pi_r m_r (\hat \beta(r)S_0-(\widecheck \delta+\mu))>0
 \end{equation}
then for any initial value $(x(0),\eta(0),r(0)) \in \mathring \Gamma \times \mathbb R^+ \times \cM$, the following statement is valid with probability 1:
 \begin{equation}\label{persS2}
   \liminf_{t \to \infty} \frac{1}{t} \int_0^t S(u) du \geq \frac{\mu}{\mu+\nu+\widecheck \beta}, 
 \end{equation}
 
 \begin{equation}\label{persAI2}
     \liminf_{t \to \infty} \frac{1}{t} \int_0^t (I(u)+A(u)) du \geq \frac{\mu+\nu+\gamma}{\widecheck{ \hat \beta}(\widecheck \beta+\gamma)}\frac{\sum_{r \in \cM} \pi_r m_r (\hat \beta(r)S_0-(\widecheck \delta+\mu))}{\sum_{r \in \cM} \pi_r m_r}.
 \end{equation}
 \end{theorem}
 
 \begin{proof}
 Assertion \eqref{persS2} can be proved in the same way as assertion \eqref{persS} in Theorem \ref{pers1}, by considering that $-(\beta_A(r)A+\beta_I(r)I)S \geq -\widecheck \beta(A+I) S$.
 Let us prove assertion \eqref{persAI2}. 
  By summing the second and third equations of \eqref{sairs3_s}, we have that

\begin{align*}
\frac{d \ln(I(t)+A(t))}{dt}&= \frac{1}{I+A}\left[(\beta_A(r)A+\beta_I(r)I)S- (\delta_A -\mu) A - (\delta_I-\mu)I\right]\\
& \geq \hat \beta(r) S-(\widecheck \delta+\mu).
\end{align*}
Following similar arguments as in \eqref{lnIA}, we obtain that
 
 \begin{equation}\label{lnAI2}
 \begin{split}
\ln (I(t)+A(t))& \geq \ln(I(0)+A(0)) + \int_0^t \hat \beta(r(u))\frac{\gamma+\mu}{\nu+\gamma+ \mu} du -\widecheck{ \hat \beta} \int_0^t \left(\frac{\gamma+\mu}{\nu+\gamma+ \mu} -S(u)\right) du -\int_0^t (\widecheck\delta+\mu) du.
\end{split}
\end{equation}
 Now, by the same steps as in \eqref{ds2}, we obtain

 \begin{equation}\label{sintdis2}
\begin{split}
(\nu+ \gamma +\mu) \int_0^t \left( \frac{\gamma + \mu}{\nu+ \gamma +\mu}-S(u)\right) du \leq S(t)-S(0) + (\widecheck \beta+\gamma) \int_0^t (I(u)+A(u)) du.
\end{split}    
\end{equation}
 By combining \eqref{lnAI2} and \eqref{sintdis2}, we have
 \begin{equation*}
\begin{split}
  \ln (I(t)+A(t))& \geq \ln(I(0)+A(0)) + \int_0^t \left(\hat \beta(r(u))\frac{\gamma+\mu}{\nu+\gamma+ \mu} -(\widecheck\delta +\mu)\right) du\\
  & -\frac{\widecheck{ \hat \beta}}{\nu+\gamma+ \mu}\left[ S(t)-S(0) +(\widecheck \beta +\gamma) \int_0^t (I(u)+A(u)) du\right].
  \end{split}
\end{equation*}
Finally, with the same arguments as in Theorem \ref{pers1}, we obtain \eqref{persAI2}. 
\end{proof}

\begin{remark}
 Let us fix $r \in \cM$. Let us consider condition \eqref{condpersdiff} and $\cR_0(r)$ in \eqref{R0}. It is easy to see that it holds
 
 \begin{equation*}
    (\hat \beta(r)S_0-(\widecheck \delta+\mu))>0 \Rightarrow \cR_0(r)>1
 \end{equation*}
\end{remark}
 
 \begin{corollary}
 Let us assume $\beta_A(r)\neq \beta_I(r)$ or $\delta_A \neq \delta_I$ in each subsystem $r=1, \ldots, M$. If
 \begin{equation*}
     \sum_{r \in \cM} \pi_r m_r (\hat \beta(r)S_0-(\widecheck \delta+\mu))>0
 \end{equation*}
then for any initial value $(x(0),\eta(0),r(0)) \in \mathring \Gamma \times \mathbb R^+ \times \cM$, the following statements hold with probability 1:
$$
\liminf_{t \to \infty} \frac{1}{t} \int_0^t I(u)du \geq \frac{\alpha}{\alpha+\delta_I+\mu} \frac{\mu+\nu+\gamma}{\widecheck{ \hat \beta}(\widecheck \beta+\gamma)}\frac{\sum_{r \in \cM} \pi_r m_r (\hat \beta(r)S_0-(\widecheck \delta+\mu))}{\sum_{r \in \cM} \pi_r m_r},$$
and
$$
\liminf_{t \to \infty} \frac{1}{t} \int_0^t A(u)du \geq \frac{\delta_I+\mu}{\alpha+\delta_I+\mu} \frac{\mu+\nu+\gamma}{\widecheck{ \hat \beta}(\widecheck \beta+\gamma)}\frac{\sum_{r \in \cM} \pi_r m_r (\hat \beta(r)S_0-(\widecheck \delta+\mu))}{\sum_{r \in \cM} \pi_r m_r}.$$
 \end{corollary}
The proof is analogous to that of Corollary \ref{cor:persI}, by invoking \eqref{persAI2}.

 \begin{remark}
 Let $\beta_A(r)=\beta_I(r):=\beta(r)$, $\delta_A=\delta_I:=\delta$, for all $r=1, \ldots, M$. From the almost surely persistence in time mean proved in Theorem \ref{pers1}, the following weak persistence follows: if $\cR_0> 1$, then for any initial value $(x(0),\eta(0),r(0)) \in \mathring \Gamma \times \mathbb R^+ \times \cM$, it holds
$$\limsup_{t \to \infty}(I(t)+ A(t)) \geq \frac{\mu+\nu+\gamma}{\widecheck \beta(\widecheck \beta+\gamma)}\frac{\sum_{r \in \cM} \pi_r m_r (\beta(r)S_0-(\delta+\mu))}{\sum_{r \in \cM} \pi_r m_r} \qquad \text{a.s.}$$

 Let $\beta_A(r)\neq \beta_I(r)$ or $\delta_A\neq\delta_I$, for all $r=1, \ldots, M$. From Theorem \ref{pers2} follows: if $ \sum_{r \in \cM} \pi_r m_r (\hat \beta(r)S_0-(\widecheck \delta+\mu))>0$
 then for any initial value $(x(0),\eta(0),r(0)) \in \mathring \Gamma \times \mathbb R^+ \times \cM$, it holds
$$\limsup_{t \to \infty}(I(t)+ A(t)) \geq  \frac{\mu+\nu+\gamma}{\widecheck{ \hat \beta}(\widecheck \beta+\gamma)}\frac{\sum_{r \in \cM} \pi_r m_r (\hat \beta(r)S_0-(\widecheck \delta+\mu))}{\sum_{r \in \cM} \pi_r m_r}, \qquad \text{a.s}.$$
 \end{remark}

 In the case $\beta_A(r)=\beta_I(r):=\beta(r)$, $\delta_A=\delta_I:=\delta$, the position of the value $\cR_0$ \eqref{r0semi} with respect to one determines the extinction or the persistence of the disease, that is $\cR_0$ is a threshold value. Thus, from Theorems \eqref{ext_switch1} and \eqref{pers1}, we obtain the following corollary:
 \begin{corollary}\label{corfraction}
 Let us assume $\beta_A(r)=\beta_I(r):=\beta(r)$, $\delta_A=\delta_I:=\delta$, and consider $\cR_0$ in \eqref{r0semi}. Then,the solution of system \eqref{sairs3_s} has the property that
 \begin{itemize}
     \item[(i)] If $\cR_0<1$,  for any initial value
 $(x(0),\eta(0),r(0)) \in \Gamma \times \mathbb R^+ \times \cM$,  the fraction of asymptomatic and infected individuals $A(t)$ and $I(t)$, respectively,
     tends to zero exponentially almost surely, that is the disease dies out with probability one;
     \item[(ii)] If $\cR_0>1$,  for any initial value
 $(x(0),\eta(0),r(0)) \in \mathring \Gamma \times \mathbb R^+ \times \cM$,  the disease will be almost surely persistent in time mean.
 \end{itemize}
 \end{corollary}
 
 \begin{remark}\label{persimplr0}
 Let us assume $\beta_A(r)\neq \beta_I(r)$ or $\delta_A\neq\delta_I$. Let us fix $r \in \cM$. Let us consider condition \eqref{condpersdiff} and $\cR_0(r)$ in \eqref{R0}. It is easy to see that it holds
 
 \begin{equation*}
    (\hat \beta(r)S_0-(\widecheck \delta+\mu))>0 \Rightarrow \cR_0(r)>1
 \end{equation*}
\end{remark}
 
 In the case $\beta_A(r)\neq \beta_I(r)$ or $\delta_A\neq\delta_I$, we find two regions, %depending on the model parameters, 
 one where the system goes to extinction almost surely, and the other where it is stochastic persistent in time mean. These two regions are not adjacent, as there is a gap between them; thus we do not have a threshold value separating these regions.\\

 In the following Section \ref{sec:limit_set}, we investigate the omega-limit set of the system. The introduction of the backward recurrence time process allows the considered stochastic system to be a piecewise deterministic Markov process \cite{li2019threshold}.
 Thus, in Section \eqref{sec:ipm}, we prove the existence of a unique invariant probability measure for this process. 
 Let us note that in the two subsequent sections, to obtain our results, we follow mainly the approaches in \cite{li2019threshold} and, like them, for simplicity, we restrict the analysis to a semi-Markov process with state space $\cM=\{1,2\}$. Hence, the external environmental conditions can switch randomly between two states, for example favorable and adverse weather conditions for the disease spread, or lockdown and less stringent distance measures, considering that the disease transmission rate is also a function of the contact rate. \\

\section{Omega-limit set}\label{sec:limit_set}

Let us assume in this section and in the subsequent one that $\cM=\{1,2\}$. \\
Let us define the omega-limit set of the trajectory starting from an initial value $x(0) \in \Gamma$ as
\begin{equation}\label{eq:omega}
    \tilde \Omega(x(0),\omega)= \bigcap_{T>0}\overline{\bigcup_{t >T} x(t,\omega,x(0))}
\end{equation}

%i.e., it is is the state the dynamical system reaches after an infinite amount of time has passed, going forward in time.
We use the notation $\tilde \Omega$ for the limit set \eqref{eq:omega} in place of the usual one $\omega$ in the deterministic dynamical systems for avoiding conflict with the element $\omega$ in the probability sample space.
In this section, it will be shown that under some appropriate condition $\tilde \Omega(x(0),\omega)$ is deterministic, i.e., it is constant almost surely and it is independent of the initial value $x(0)$.
Let us consider the following assumption:\\

\textbf{(H2)} For some $r \in \cM$, there exists a unique and globally asymptotically stable endemic equilibrium $x^*_r=(S^*_r,A^*_r,I^*_r)$ for the corresponding system of \eqref{sairs3_s} in the state $r$.\\

Let us note that when $\beta_A(r)=\beta_I(r):=\beta(r)$ and $\delta_A=\delta_I:=\delta$, in the case of persistence in time mean, by (ii) of Proposition \ref{prop:ROsemi} the condition $\cR_0>1$ implies that there exists at least one state $r$ such that $ \pi_r m_r (\beta(r) S_0-(\delta+\mu)) >0$, i.e. $\cR_0(r)>1$. 
%By Lemma \ref{ex_ee}, if $\cR_0(r)>1$ there exists a unique endemic equilibrium $x_r^*$ in $\mathring \Gamma$ that is locally asymptotically stable.Moreover, 
By Theorem \ref{thm:globeq}, $x_r^*$ is globally asymptotically stable in $\mathring \Gamma$. %when the infection rate of the asymptomatic individuals are equal to that of the symptomatic ones, as well as their recovery rates. 
Thus, if $\cR_0>1$ condition \textbf{(H2)} is satisfied.\\
When $\beta_A(r)\neq \beta_I(r)$ or $\delta_A\neq\delta_I$, by Theorem \ref{pers2}, if equation \eqref{condpersdiff} holds, then there exists at least one state $r$ such that  $(\hat \beta(r)S_0-(\widecheck \delta+\mu))>0$; from Remark
\ref{persimplr0} this implies that $\cR_0(r)>1$.
By Theorem \eqref{thm:globdiff}, we have the global asymptotic stability of $x_r^*$ if $\cR_0(r)>1$  under the additional condition $\beta_A(r) < \delta_I$. However, it is easy to see that if this last condition is verified, $(\hat \beta(r)S_0-(\widecheck \delta+\mu))>0$ cannot be valid. Thus, if we need \textbf{(H2)}, we can suppose it holds for a state for which $(\hat \beta(r)S_0-(\widecheck \delta+\mu))>0$ is not verified. Indeed, we remember that this last condition is only sufficient to have $\cR_0(r)>1$ but not necessary.  \\

Now, let us recall some concepts on the Lie algebra of vector fields \cite{benaim2015qualitative,jurdjevic1997geometric} that we need for the next results.
Let $w(y)$ and $z(y)$ be two vector fields on $\mathbb R^3$. The Lie bracket $[w,z]$ is also a vector field given by
$$[w,z]_j(y)=\sum_{k=1}^3 \left( w_k \frac{\partial z_j}{\partial y_k}(y)-z_k\frac{\partial w_k}{\partial y_k}(y) \right), \qquad j=1,2,3$$
Assumption \textbf{(H3)}:\\
A point $x=(S,A,I) \in \mathbb{R}^3_+$ is said to satisfy the \emph{Lie bracket condition}, if vectors $u_1(x)$, $u_2(x)$,
%\ldots, $u_M(x)$, %se considero N stati 
$[u_i,u_j](x)_{i,j \in \cM}$, $[u_i,[u_j,u_k]](x)_{i,j,k \in \cM}, \ldots$, span the space $\mathbb R^3$, where for each $r \in \cM$,

\begin{equation}\label{lie}
    u_r(x) = \left(
\begin{matrix}
      \mu -\bigg(\beta_A(r)A + \beta_I(r)I\bigg)S -(\mu + \nu+ \gamma) S + \gamma (1-A-I)\\
\bigg(\beta_A(r)A + \beta_I(r)I(t)\bigg)S -(\alpha + \delta_A +\mu) A\\
   \alpha A -(\delta_I +\mu)I
   \end{matrix}
    \right).
\end{equation}

Without loss of generality, we can assume that condition \textbf{(H2)} holds for $r=1$.
\begin{theorem}\label{limitset}
Suppose that system \eqref{sairs3_s} is persistent in time mean and the hypothesis \textbf{(H2)} holds. Let us denote by
$\xi^r_t(x(0))$ the solution of system \eqref{sairs3_s} in the state $r$ with initial value \ste{$x(0) \in \mathring \Gamma$}, and let
$$\Psi= \bigg\{ (S,A,I)=\xi^{e_k}_{t_k} \circ \ldots \circ \xi^{e_1}_{t_1}(x_1^*) : t_1, \ldots, t_k \geq 0 \quad \text{and} \quad e_1, \ldots, e_k \in \cM, k \in \mathbb N  \bigg\}.$$
Then, the following statements are valid:
\begin{itemize}
    \item[(a)] The closure $\bar \Psi$ is a subset of the omega-limit set $\tilde \Omega(x(0),\omega)$ with probability one.
    \item[(b)] If there exists a point $x_*:=(S_*,A_*,I_*) \in \Psi$ satisfying the condition \textbf{(H3)}, then $\Psi$ absorbs all positive solutions, that is for any initial value $x(0) \in \mathring \Gamma$, the value 
    $$\hat T(\omega)=\inf\left\{ t >0: x(s,\omega,x(0)) \in \Psi , \forall s > t \right\}$$
    is finite outside a $\mathbb P$-null set. Consequently, $\bar \Psi$ is the omega-limit set $\tilde \Omega(x(0),\omega)$ for any $x(0) \in \mathring \Gamma$ with probability one. 
\end{itemize}
\end{theorem}
The proof of the Theorem \ref{limitset} follows by similar arguments to that of \cite[Thm 9]{li2019threshold}, thus we omit it.

\section{Invariant probability measure}\label{sec:ipm}

In this section, we will prove the existence of an invariant probability measure for the homogeneous Markov process $\{(x(t), \eta(t), r(t)), t \geq 0\}$ on the state space $$\cH= \mathring\Gamma \times \mathbb R_+ \times \cM.$$
Following \cite{li2019threshold}, we introduce a metric $\hbar(\cdot, \cdot)$ on the state space $\cH$:

\begin{equation}\label{hmetric}
\hbar \left((x_1,s_1,i), (x_2,s_2,j)\right)= \sqrt{|x_1-x_2|^2+|s_1-s_2|^2}+o(i,j),
\end{equation}
where 
$$o(i,j)=\begin{cases}
0, & \text{if} \quad i=j,\\
1, & \text{if} \quad i \neq j.
\end{cases}$$
Hence, $(X,\hbar(\cdot, \cdot),\cB(\cH))$ is a complete separable metric space, where $\cB(\cH)$ is the Borel $\sigma$-algebra on $\cH$. %i.e., the smallest $\sigma-$algebra which contains every open set

To ensure the existence of an invariant probability measure on $\cH$, we use the following exclusion principle.

\begin{lemma}[see \cite{stettner1986existence}]\label{existmu}\label{exclusion}
Let $\Phi=\left\{ \Phi_t, t \geq 0 \right\}$ be a Feller process with state space $(X,\cB(X))$.  Then either

\begin{itemize}
\item[a)] there exists an invariant probability measure on $X$, or
\item[b)] for any compact set $C \subset X$,
%$$\limsup_{t \to \infty} 
$$\lim_{t \to \infty}\sup_{\kappa} \frac{1}{t} \int_0^t \left( \int_X \bP(u,x,C) \kappa( {\mathrm d}x)\right){\mathrm d} u=0,  $$
where the supremum is taken over all initial distributions $\kappa$ on the state space $X$, $x \in X$ is the initial condition for the process $\Phi_t$, and $\bP(t,x,C)=\bP_x(\Phi_t \in C)$ is the transition probability function.
\end{itemize}
\end{lemma}
To prove the Feller property of the process $\{(x(t), \eta(t), r(t)), t \geq 0\}$, we need the following lemma. %\cite{li2019threshold}.

\begin{lemma}\label{lemma:hmetric}
Let $\hbar(\cdot, \cdot)$ be the metric defined in \eqref{hmetric}. Then, for any $T>0$ and $\eps >0$, we have

\begin{equation}\label{Phmetric}
\bP \left\{ \max_{0 \leq t \leq T} \hbar \left( (x(t,x_1), \eta(t), r(t)), (x(t,x_2), \eta(t), r(t))\right) \geq \eps \right\} \to 0    
\end{equation}
as $|x_1-x_2| \to 0$, where $(x_1,\eta(0),r(0)),$ $(x_2,\eta(0),r(0)) \in \cH$ denote any two given initial values of the process  $\{(x(t), \eta(t), r(t)), t \geq 0\}$. 
\end{lemma}

\begin{proof}
By considering \eqref{lie},  it is easy to see that 

\begin{equation}\label{dx1x2}
d\left( x(t,x_1)-x(t,x_2)\right)= \left( u_{r(t)}(x(t,x_1)-u_{r(t)}(x(t,x_2))\right) dt.   
\end{equation}
Applying the It\^o formula to the function $|x(t,x_1)-x(t,x_2)|^2$, we have

\begin{equation}\label{itoE}
\bE |x(t,x_1)-x(t,x_2)|^2 = |x_1-x_2|^2+2 \bE\left[ \int_0^t\langle x(s,x_1)-x(s,x_2), u_{r(s)}(x(s,x_1))-u_{r(s)}(x(s,x_2)  \rangle\; ds\right]   
\end{equation}
For ease of notation, let us define

$$G(s,x_k)=(\beta_A(r(s))A(s,x_k)+\beta_I(r(s))I(s,x_k)), \qquad k=1,2.$$
Now, we have that
\begin{align*}
&\langle x(s,x_1)-x(s,x_2), u_{r(s)}(x(s,x_1))-u_{r(s)}(x(s,x_2)  \rangle\\
&=(S(s,x_1)-S(s,x_2))(G(s,x_1)S(s,x_1)-G(s,x_2)S(s,x_2))- (\mu+\nu+\gamma)(S(s,x_1)-S(s,x_2))^2 \\
&-\gamma (S(s,x_1)-S(s,x_2))(A(s,x_1)-A(s,x_2))-\gamma(S(s,x_1)-S(s,x_2))(I(s,x_1)-I(s,x_2))\\
&+(A(s,x_1)-A(s,x_2))(G(s,x_1)S(s,x_1)-G(s,x_2)S(s,x_2))-(\alpha+\delta_A+\mu)(A(s,x_1)-A(s,x_2))^2\\
&+ \alpha (A(s,x_1)-A(s,x_2))(I(s,x_1)-I(s,x_2))-(\delta_I+\mu)(I(s,x_1)-I(s,x_2))^2. 
\end{align*}
Since
$$\gamma(S(s,x_1)-S(s,x_2))(A(s,x_1)-A(s,x_2)) \leq \frac{\gamma}{2}(S(s,x_1)-S(s,x_2))^2 + \frac{\gamma}{2}(A(s,x_1)-A(s,x_2))^2,$$
$$\gamma(S(s,x_1)-S(s,x_2))(I(s,x_1)-I(s,x_2)) \leq \frac{\gamma}{2}(S(s,x_1)-S(s,x_2))^2 + \frac{\gamma}{2}(I(s,x_1)-I(s,x_2))^2,$$
and
$$ \alpha (A(s,x_1)-A(s,x_2))(I(s,x_1)-I(s,x_2)) \leq \frac{\alpha}{2}(A(s,x_1)-A(s,x_2))^2+\frac{\alpha}{2}(I(s,x_1)-I(s,x_2))^2,$$
we can write
\begin{align}\label{lrleq}
\begin{split}
&\langle x(s,x_1)-x(s,x_2), u_{r(s)}(x(s,x_1))-u_{r(s)}(x(s,x_2)  \rangle\\
&\leq (\mu+\nu+2\gamma)(S(s,x_1)-S(s,x_2))^2+\left(\frac{\gamma}{2}+\frac{3\alpha}{2}+\delta_A+\mu\right)(A(s,x_1)-A(s,x_2))^2 \\
&+\left(\delta_I+\mu +\frac{\gamma}{2}+\frac{\alpha}{2}\right)(I(s,x_1)-I(s,x_2))^2\\
&+|(S(s,x_1)-S(s,x_2))(G(s,x_1)S(s,x_1)-G(s,x_2)S(s,x_2))|\\
&+|(A(s,x_1)-A(s,x_2))(G(s,x_1)S(s,x_1)-G(s,x_2)S(s,x_2))|
\end{split}
\end{align}
Now, 
\begin{align}\label{SGleq}
\begin{split}
&|(S(s,x_1)-S(s,x_2))(G(s,x_1)S(s,x_1)-G(s,x_2)S(s,x_2))|\\
&\leq \widecheck \beta (S(s,x_1)-S(s,x_2))\left[ (S(s,x_1)-S(s,x_2))(A(s,x_1)+I(s,x_1))\right.\\
& \left. + S(s,x_2)(A(s,x_1)+I(s,x_1)-(A(s,x_2)+I(s,x_2))) \right]\\
&\leq \widecheck \beta (S(s,x_1)-S(s,x_2))^2 + \widecheck \beta (S(s,x_1)-S(s,x_2))(A(s,x_1)+I(s,x_1)-(A(s,x_2)+I(s,x_2)))\\
&\leq \widecheck \beta(S(s,x_1)-S(s,x_2))^2+ \frac{\widecheck \beta}{2}\left[ (S(s,x_1)-S(s,x_2))^2+(A(s,x_1)-A(s,x_2))^2\right]\\
&+\frac{\widecheck \beta}{2}\left[ (S(s,x_1)-S(s,x_2))^2+(I(s,x_1)-I(s,x_2))^2\right],
\end{split}
\end{align}
and, similarly
\begin{align}\label{AGleq}
\begin{split}
&|(A(s,x_1)-A(s,x_2))(G(s,x_1)S(s,x_1)-G(s,x_2)S(s,x_2))|\\
&\leq \widecheck \beta (A(s,x_1)-A(s,x_2))\left[ (S(s,x_1)-S(s,x_2))(A(s,x_1)+I(s,x_1))\right.\\
& \left. + S(s,x_2)(A(s,x_1)+I(s,x_1)-(A(s,x_2)+I(s,x_2))) \right]\\
&\leq \frac{\widecheck \beta}{2}((A(s,x_1)-A(s,x_2))^2+(S(s,x_1)-S(s,x_2))^2) \\
&+ \widecheck \beta ((A(s,x_1)-A(s,x_2))S(s,x_2)(A(s,x_1)+I(s,x_1)-(A(s,x_2)+I(s,x_2)))\\
&\leq \frac{\widecheck \beta}{2}((A(s,x_1)-A(s,x_2))^2+(S(s,x_1)-S(s,x_2))^2) + \widecheck \beta((A(s,x_1)-A(s,x_2))^2)\\
&+ \frac{\widecheck \beta}{2}((A(s,x_1)-A(s,x_2))^2+(I(s,x_1)-I(s,x_2))^2).
\end{split}
\end{align}
Now, substituting \eqref{SGleq} and \eqref{AGleq} into \eqref{lrleq} yields 

\begin{align}\label{maxT}
\begin{split}
&\langle x(s,x_1)-x(s,x_2), u_{r(s)}(x(s,x_1))-u_{r(s)}(x(s,x_2)  \rangle\\
&\leq K|x(s,x_1)-x(s,x_2)|^2,
\end{split}
\end{align}
where $K=\max\{K_1,K_2,K_3\}$, with
$$K_1=\frac{5 \widecheck \beta}{2} + (\mu+\nu+2\gamma),$$
$$K_2=\frac{5 \widecheck \beta}{2} +\left( \frac{\gamma}{2}+\frac{3\alpha}{2}+\delta_A +\mu\right),$$
and
$$K_3=\widecheck \beta +\delta_I+\mu+ \frac{\gamma}{2} +\frac{\alpha}{2}.$$
From \eqref{itoE} and \eqref{maxT}, following similar steps as in \cite[Lemma 14]{li2019threshold}, we have that
\begin{equation}\label{Eto0}
\int_0^T  \bE\left|x(s,x_1)-x(s,x_2)\right|^2 ds \leq |x(s,x_1)-x(s,x_2)|^2 \int_0^T \exp(2Ks)ds \to 0  
\end{equation}
as $|x(s,x_1)-x(s,x_2)| \to 0$. Moreover, from \eqref{dx1x2}, it follows that
\begin{equation}\label{maxx1x2}
\max_{0 \leq t \leq T} \left| x(s,x_1)-x(s,x_2)\right|  \leq  \left| x(s,x_1)-x(s,x_2)\right|+ \int_0^T \left| u_{r(s)}(x(s,x_1))-u_{r(s)}(x(s,x_2)\right| ds.
\end{equation}
For any $\eps >0$, from Markov inequality and similar steps as in \cite[Lemma 14]{li2019threshold}, we obtain
\begin{align}\label{MI}
\begin{split}
&\bP\left\{  \int_0^T \left| u_{r(s)}(x(s,x_1))-u_{r(s)}(x(s,x_2)\right| ds \geq \eps \right\}  \\
&\leq \frac{T}{\eps^2} \bE\left[ \int_0^T  |u_{r(s)}(x(s,x_1))-u_{r(s)}(x(s,x_2)|^2 ds \right]
\end{split}
\end{align}
By similar arguments to \eqref{lrleq}, \eqref{SGleq} and \eqref{AGleq}, one can find a positive number $\bar K$ such that
$$\left| u_{r(s)}(x(s,x_1))-u_{r(s)}(x(s,x_2)\right|^2 \leq \bar K \left| x(s,x_1)-x(s,x_2)\right|^2. $$
Substituting this inequality into \eqref{MI}, by Fubini's theorem we have 
\begin{equation*}
\bP\left\{  \int_0^T \left| u_{r(s)}(x(s,x_1))-u_{r(s)}(x(s,x_2)\right| ds \geq \eps \right\} \leq \frac{\bar K T}{\eps^2}\int_0^T \bE \left| x(s,x_1)-x(s,x_2)\right|^2 ds. 
\end{equation*}
By \eqref{Eto0}, then we get
\begin{equation}\label{Pto0}
\bP\left\{  \int_0^T \left| u_{r(s)}(x(s,x_1))-u_{r(s)}(x(s,x_2)\right| ds \geq \eps \right\} \to 0
\end{equation}
as $|x(s,x_1)-x(s,x_2)| \to 0$. Definitively, by combining \eqref{maxx1x2} with \eqref{Pto0}, we obtain the claim \eqref{Phmetric}.
\end{proof}

\begin{lemma}\label{feller}
The Markov process $\{(x(t),\eta(t),r(t), t \geq 0\}$ is Feller. 
\end{lemma}
The proof requires the claim of Lemma \ref{lemma:hmetric} and it is analogous to that of \cite[Lemma 15]{li2019threshold}, thus we omit it.\\
%Stettner: On the existence and uniqueness of invariant measure for continuous time Markov processes

At this point, we can prove the existence of an invariant probability measure by using Lemma \eqref{exclusion}.
\begin{theorem}\label{ipm}
Suppose that system \eqref{sairs3_s} is persistent in time mean, then the Markov process $\{(x(t),\eta(t),r(t), t \geq 0\}$ has an invariant probability measure $\kappa^*$ on the state space $\cH$.
\end{theorem}
\begin{proof}
Let us consider the process $\{(x(t),\eta(t),r(t)), t\geq 0\}$ on a larger state space
$$\tilde\cH=\tilde{ \Gamma} \times \mathbb R_+ \times \cM,$$
where $\tilde {\Gamma}= \Gamma \setminus \{(S,A,I) : A+I=0 \}$. By Lemma \ref{exclusion} and \ref{feller}, we can prove the existence of an invariant probability measure $\kappa^*$ for the Markov process $h(t)=\{(x(t),\eta(t),r(t)), t\geq 0\}$ on $\tilde \cH$, provided that a compact subset $C \subset \tilde \cH$ exists such that  
\begin{equation}\label{liminfP}
 \liminf_{t \to \infty} \frac{1}{t}\int_0^t \left(\int_{\tilde \cH} \bP(u,h,C)\kappa(dh) \right) du=  \liminf_{t \to \infty} \frac{1}{t}\int_0^t \bP(u,h_0, C)du>0,
\end{equation}
for some initial distribution $\kappa=\delta_{h_0}$ with $h_0 \in \tilde \cH$, where $\delta.$ is the Dirac function. Once the existence of $\kappa^*$ is proved, we can easily see that $\kappa^*$ is also an invariant probability measure of $\{(x(t),\eta(t),r(t)), t\geq 0\}$ on $\cH$. Indeed, it is easy to prove that for any initial value $(x(0),\eta(0),r(0)) \in \mathring \Gamma \times \mathbb R^+ \times \cM$, %parto dall'interno perché sono nella regione di persisitenza e se parto da I,A=0 non sarei persistente
the solution $x(t)$ of system \eqref{sairs3_s} does not reach the boundary $\partial \tilde \Gamma$ %= \partial Gamma 
when the system is persistent in time mean. Consequently, $k^*(\partial \tilde \Gamma \times \mathbb R_+ \times \cM)=0$, which implies therefore that $k^*$ is also an invariant probability measure on $\cH$. Thus, to complete the proof we just have to find a compact subset $C \subset \tilde \cH$ satisfying \eqref{liminfP}.
Hereafter, in the proof, we assume, without loss of generality that $\eta(0)=0$. Since system \eqref{sairs3_s} is persistent in time mean, there exists a constant $\iota >0$ such that
$$\liminf_{t \to \infty} \frac{1}{t} \int_0^t (I(u)+A(u)) du \geq \iota \qquad \text{a.s.}$$
Following similar steps as in \cite[Theorem 16]{li2019threshold}, we obtain
\begin{equation}\label{Piota}
\liminf_{t \to \infty}  \int_0^t \bP \left\{ I(u)+A(u) \geq \frac{\iota}{2}\right\} du \geq \frac{\iota}{2}.   
\end{equation}
Now, we have to verify that $\bP\{\eta(u)>K \}< \eps$ holds for any $u \in \mathbb R_+$, where $\eps$ is a positive conctant such that $\eps < \frac{\iota}{4}$ and $K>0$ is a sufficiently large constant such that $\max_{i \in \cM}\left\{ 1-F_i(K)\right\} < \eps/2$. To do this, we follow similar steps as in \cite[Theorem 16]{li2019threshold}. This implies that
$$\liminf_{t \to \infty}  \int_0^t \bP \left\{ I(u)+A(u) \geq \frac{\iota}{2}\right\} du \leq \eps +\frac{1}{t} \int_0^t  \bP \left\{ I(u)+A(u) \geq \frac{\iota}{2}, \eta(u) \leq K \right\} du.$$
Combining this with \eqref{Piota}, we have
$$\liminf_{t \to \infty} \frac{1}{t} \int_0^t  \bP \left\{ I(u)+A(u) \geq \frac{\iota}{2}, \eta(u) \leq K \right\} du \geq \frac{\iota}{4}.$$
Let us consider the compact set $C= \cD\times [0,K] \times \cM \in \tilde \cH,$ where the set $\cD$ is 
$$\left\{ (S,A,I) \in \Gamma : 0 \leq S+A+I \leq 1, I+A \geq \frac{\iota}{2}\right\}.$$
Then, it follows that
\begin{align*}
\liminf_{t \to \infty} \frac{1}{t} \int_{0}^t  \bP(u,h_0,C) du &=   \liminf_{t \to \infty} \frac{1}{t} \int_{0}^t \bP \left\{ x(u, \omega, x(0)) \in \cD, \eta(u) \leq K \right\} du \\
&=  \liminf_{t \to \infty} \frac{1}{t} \int_{0}^t \bP \left\{ (I(u)+A(u)) \geq \frac{\iota}{2}, \eta(u) \leq K   \right\} \\
&\geq \frac{\iota}{4}.
\end{align*}
Since $\iota>0$, we arrive to our claim.
\end{proof}
Now, we show that the invariant probability measure $\kappa^*$ is unique by using the property of Harris recurrence and positive Harris recurrence (for standard definitions see, e.g., \cite{meyn1993stability,li2019threshold}). 
\begin{proposition}
 Suppose that system \eqref{sairs3_s} is persistent in time mean and that assumptions \textbf{(H2)}-\textbf{(H3)} hold, then the Markov process $\{(x(t),\eta(t),r(t)), t\geq 0\}$ is positive Harris recurrent. Thus, the invariant probability measure $\kappa^*$ of $\{(x(t),\eta(t),r(t)), t\geq 0\}$ on the state space $\cH$ is unique.  
\end{proposition}
  \begin{proof}
   The Harris recurrence of the process $\{(x(t),\eta(t),r(t)), t\geq 0\}$ follows by similar steps as in \cite[Lemma 20]{li2019threshold}. Thus, from Theorem \ref{ipm}, we can conclude that the process is positive Harris recurrent. Therefore, the invariant probability measure $\kappa^*$ of $\{(x(t),\eta(t),r(t)), t\geq 0\}$ on $\cH$ is unique \cite{getoor1980transience,meyn1993stability}.
   \end{proof}

%\stef{uniqueness}: if the Markov process is Harris recurrent (Lemma 20 \cite{li2019threshold}), then a unique invariant measure exists. %If the invariant measure is finite, then it may be normalized to be a probability measure, in this case the process is called positive Harris recurrent (Stability of Markov process III)

\section{Numerical experiments}\label{num_exp}
%\textbf{Esempio con switching markoviano in cui la Q dipende dalla traiettoria.}

In this section, we provide some numerical investigations to verify the theoretical results. We assume that the conditional holding time distribution $F_i(\cdot)$
of the semi-Markov process in state $i$ is a gamma distribution $\Gamma(k_i,{\theta_i})$, $i=1,\ldots, M$, whose probability density function is

$$f_i(t; k_i,\theta_i )=\frac{\theta_i^{k_i} t^{k_i-1} e^{-\theta_i t}}{\Gamma(k_i)}, \qquad \text{for} \quad t>0,$$
and the cumulative distribution function is given by
$$F(t;k_i,\theta_i)=\int_0^t f(u;k_i,\theta_i) du, $$
where \ste{$k_i >0$}, $\theta_i >0$, and $\Gamma(\cdot)$ is the complete gamma function. Let us note that if $k_i=1$, the gamma distribution becomes the exponential distribution with parameter $\theta_i$.
 
We discuss the almost sure extinction and persistence in time mean of the system in the following cases. 

\begin{figure}[h]
    \centering
    \begin{subfigure}{0.49\textwidth}
        \centering
        \includegraphics[width=0.9\textwidth]{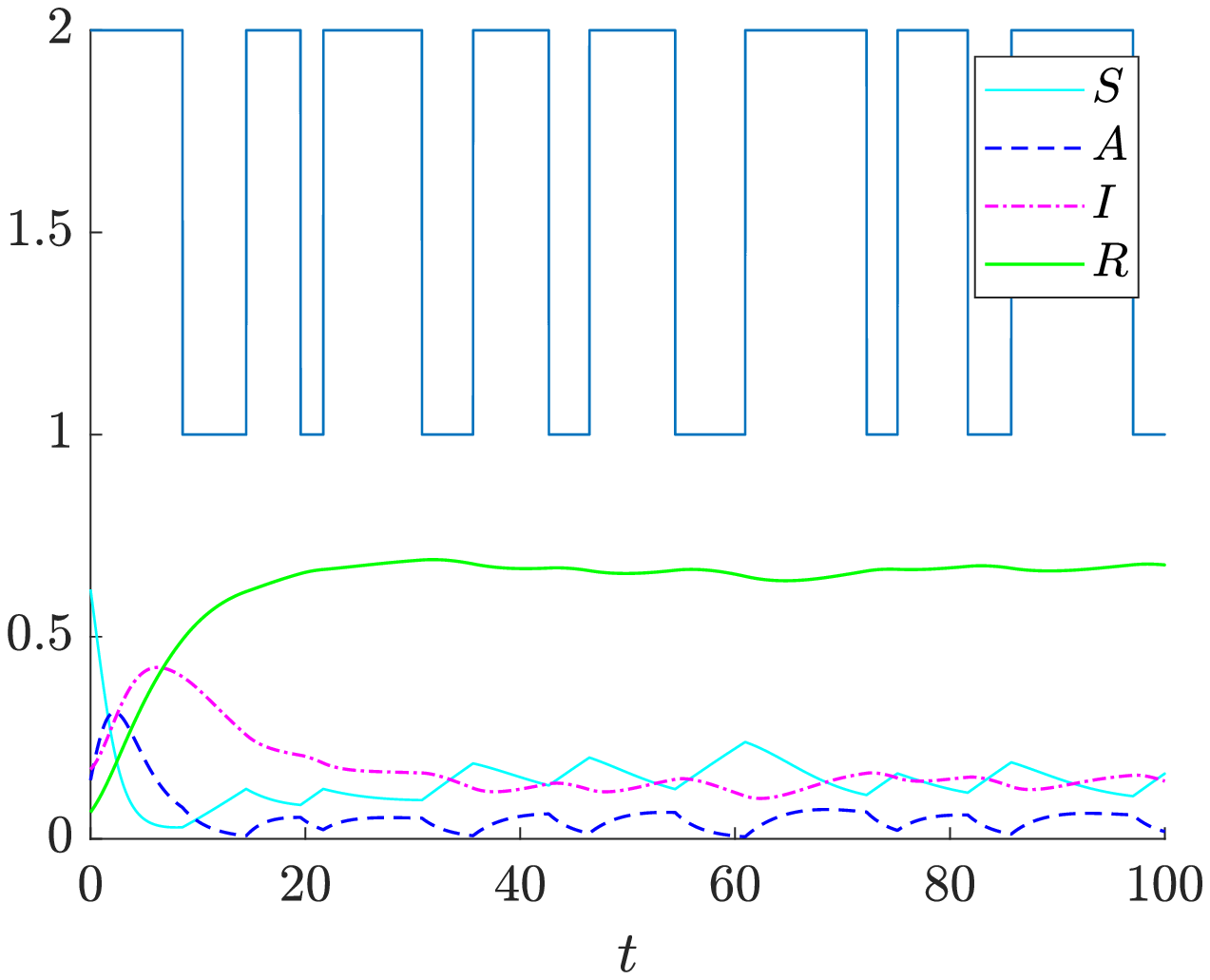}\put(-210,150){a)}
       % \caption{}\label{fig:pers1}
    \end{subfigure}\hfill
    \begin{subfigure}{0.49\textwidth}
        \centering
        \includegraphics[width=0.9\textwidth]{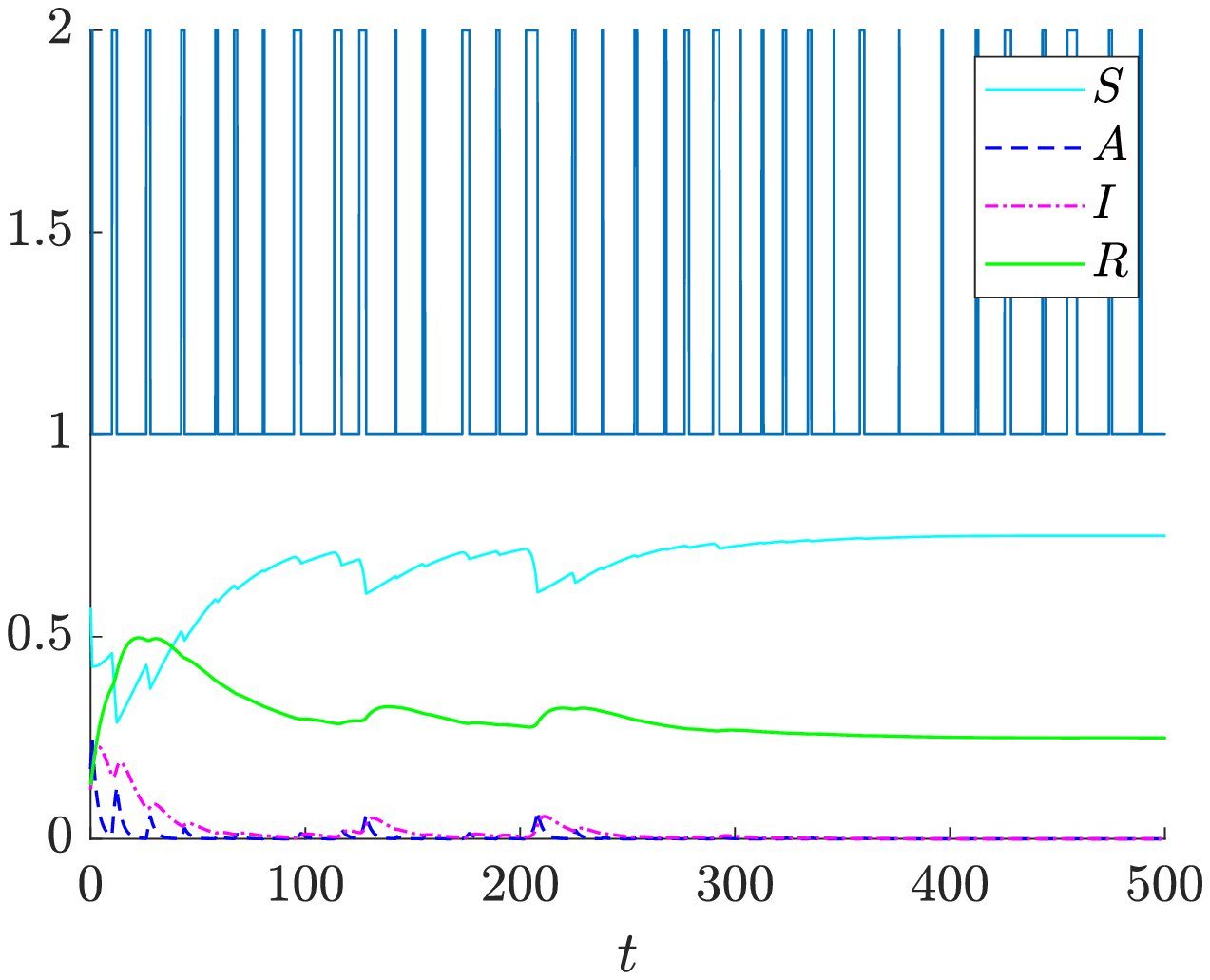}\put(-210,150){b)}
        %\caption{}\label{fig:pers2}
    \end{subfigure} 
    \caption{Dynamical behaviour of system \eqref{sairs_s} and the semi-Markov process $r(t)$. The parameter values are:  $\beta_A(1)=0.004$, $\beta_I(1)=0.008$, $\beta_A(2)=0.97$, $\beta_I(2)=0.99$, $\delta_A=0.105$, $\delta_I=0.1$, \ste{$\mu=1/(60*365)$}, $\alpha=0.3$, $\gamma=0.03$, and $\nu=0.01$. The parameters of $F_1$ and $F_2$ are respectively: a) $k_1=6$, $\theta_1=0.8$ and $k_2=12$, $\theta_2=0.8$, b) $k_1=15$, $\theta_1=0.8$ and $k_2=2$, $\theta_2=0.8$.}
    \label{fig:PersVax}
    \end{figure}
    
  \begin{figure}[h]
    \centering
    \includegraphics[width=0.5\textwidth]{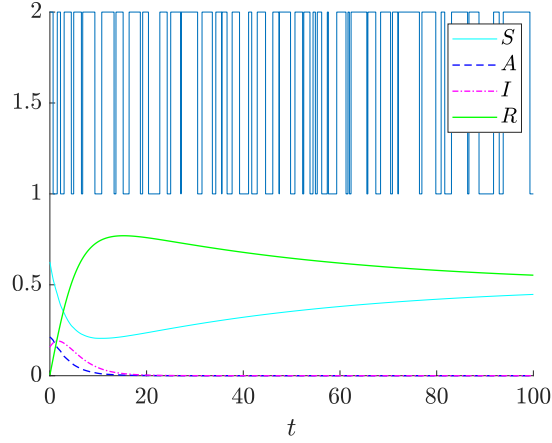}
    \caption{Dynamical behaviour of system \eqref{sairs_s} and the semi-Markov process $r(t)$. The parameter values are:  $\beta_A(1)=0.55$, $\beta_I(1)=0.5$, $\beta_A(2)=0.68$, $\beta_I(2)=0.58$, $\delta_A=0.3$, $\delta_I=0.4$, \ste{$\mu=1/(60*365)$}, $\nu=0.01$, $\alpha=0.5$, $\gamma=0.01$.  The parameters of $F_1$ and $F_2$ are $k_1=0.9$, $\theta_1=0.8$ and $k_2=2.5$, $\theta_2=0.8$, respectively}\label{fig:comparethr}
    \end{figure}

\emph{Case 1: $\beta_A(r)\neq\beta_I(r)$, $\delta_A(r)\neq\delta_I(r)$.} We consider two states: $r_1$, in which the epidemic dies out, and $r_2$ in which it persists. 
In Fig.\ref{fig:PersVax} a), we consider $F_1$, $F_2$ with parameters $k_1=6$, $\theta_1=0.8$ and $k_2=12$, $\theta_2=0.8$, respectively. Consequently, the respective mean sojourn times are $m_1=7.5$ and $m_2=15$. 
The other parameters are: $\beta_A(1)=0.004$, $\beta_I(1)=0.008$, $\beta_A(2)=0.97$, $\beta_I(2)=0.99$, $\delta_A=0.105$, $\delta_I=0.1$, \ste{$\mu=1/(60*365)$}, $\alpha=0.3$, $\gamma=0.03$, and $\nu=0.01$.
 We have that $$\sum_{r \in \cM} \pi_r m_r (\hat \beta(r)S_0-(\widecheck \delta+\mu))=4.283>0,$$ thus the whole system is stochastically persistent in time mean.
In Fig.\ref{fig:PersVax} b), we consider $F_1$, $F_2$ with parameters $k_1=15$, $\theta_1=0.8$ and $k_2=2$, $\theta_2=0.8$, respectively; we have $m_1=18.75$, $m_2=2.5$. The other parameters are the same as in a). In this case, $$\sum_{r \in \cM} \pi_r m_r (\widecheck \beta(r)S_0-(\hat \delta+\mu))=-0.0782<0,$$ and the epidemic will go extinct almost surely, in the long run. Thus, we can see the relevant role played by the mean sojourn times, indeed in the two figures the parameters are the same, what changes is that in Fig.\ref{fig:PersVax} a) the mean sojourn time in the persistent state is higher than that in the state of extinction, while in Fig.\ref{fig:PersVax} b) the vice versa occurs.

\emph{Case 2: $\beta_A(r)\neq\beta_I(r)$, $\delta_A(r)\neq\delta_I(r)$.} Here, we want to compare the two sufficient conditions for the almost sure extinction \eqref{condextdiff1} and \eqref{condext2}. In Fig.\ref{fig:comparethr}, we consider again two states $r_1$, in which the epidemic dies out, and $r_2$ in which it persists. The parameters are: $\beta_A(1)=0.55$, $\beta_I(1)=0.5$, $\beta_A(2)=0.68$, $\beta_I(2)=0.58$, $\delta_A=0.3$, $\delta_I=0.4$, \ste{$\mu=1/(60*365)$}, $\nu=0.01$, $\alpha=0.5$, $\gamma=0.01$.
Let us consider $F_1$ and $F_2$ with $k_1=0.9$, $\theta_1=0.8$ and $k_2=2.5$, $\theta_2=0.8$, respectively. Hence, $m_1=1.125$ and $m_2=3.125$. We have

$$\sum_{r \in \cM} \pi_r m_r (\widecheck \beta(r)S_0-(\hat \delta+\mu))=0.05>0 \qquad \text{and} \qquad \sum_{r \in \cM}  \pi_r m_r \lambda_1(B(r)+B(r)^T)= -0.1959<0$$

Thus, in this case condition \eqref{condextdiff1} is not satisfied but \eqref{condext2} does. Vice versa, in Fig.\ref{fig:PersVax} b), we have the opposite case, that is condition \eqref{condextdiff1} is less than zero, while 
    \eqref{condext2} is equal to $1.0084$, hence greater than zero. By Theorems \ref{thm:extdiff1} and Theorem \ref{thm:extdiff2}, we have the almost sure extinction in both cases, as we can also see from Figs. \ref{fig:comparethr} and \ref{fig:PersVax} b).
    
      \begin{figure}[h]
 \centering%\put(-220,160){a)}\put(-220,70){b)}
    %\subfloat[\label{fig:AQ1}]{%
    \begin{minipage}{0.4\textwidth}
    \includegraphics[width=\textwidth,height=5cm]{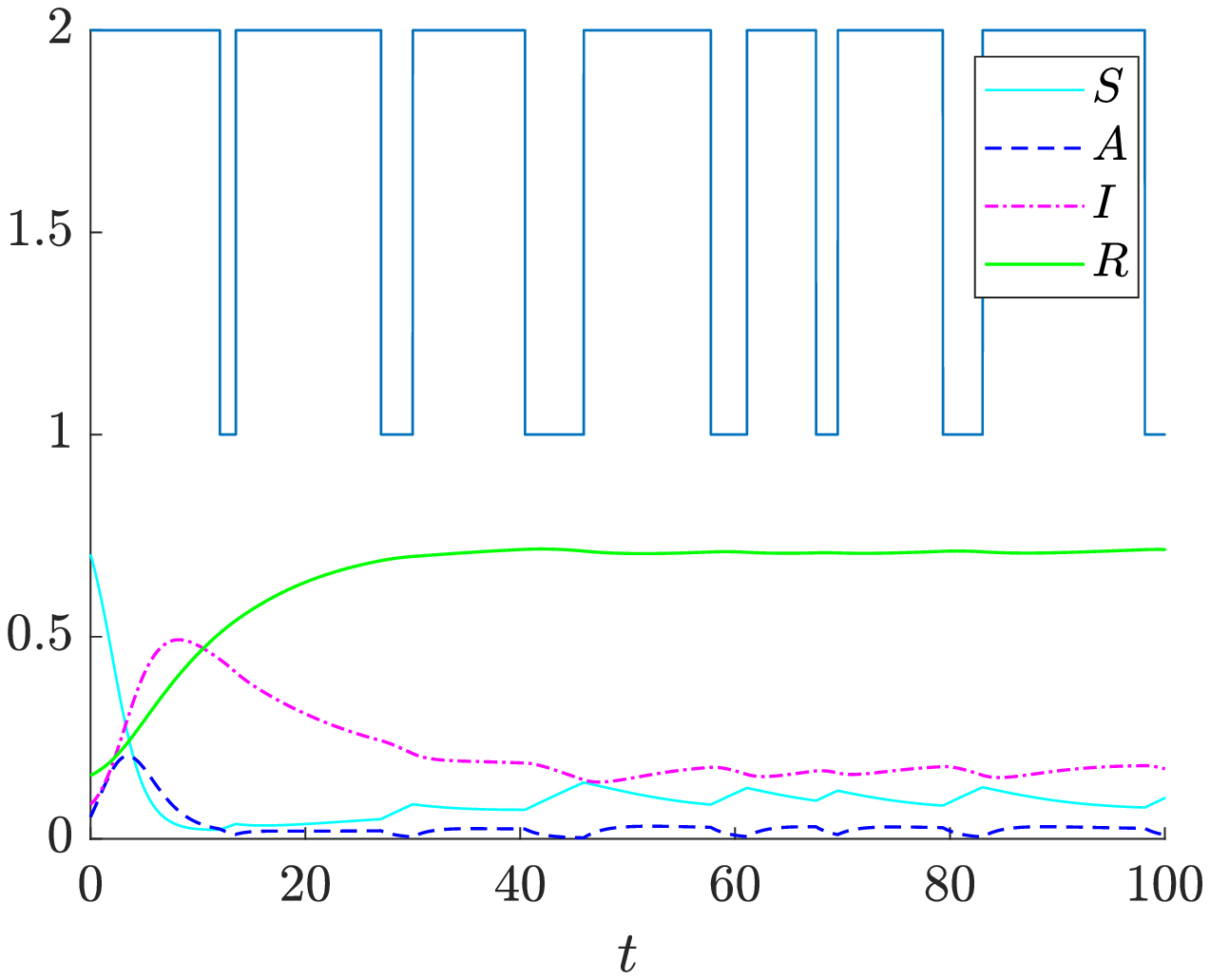}\put(-210,130){a)}
\end{minipage}
    %}
   \hskip4mm
    %\subfloat[\label{fig:AQ2}]{%
   \begin{minipage}{0.4\textwidth}
    \includegraphics[width=\textwidth,height=5cm]{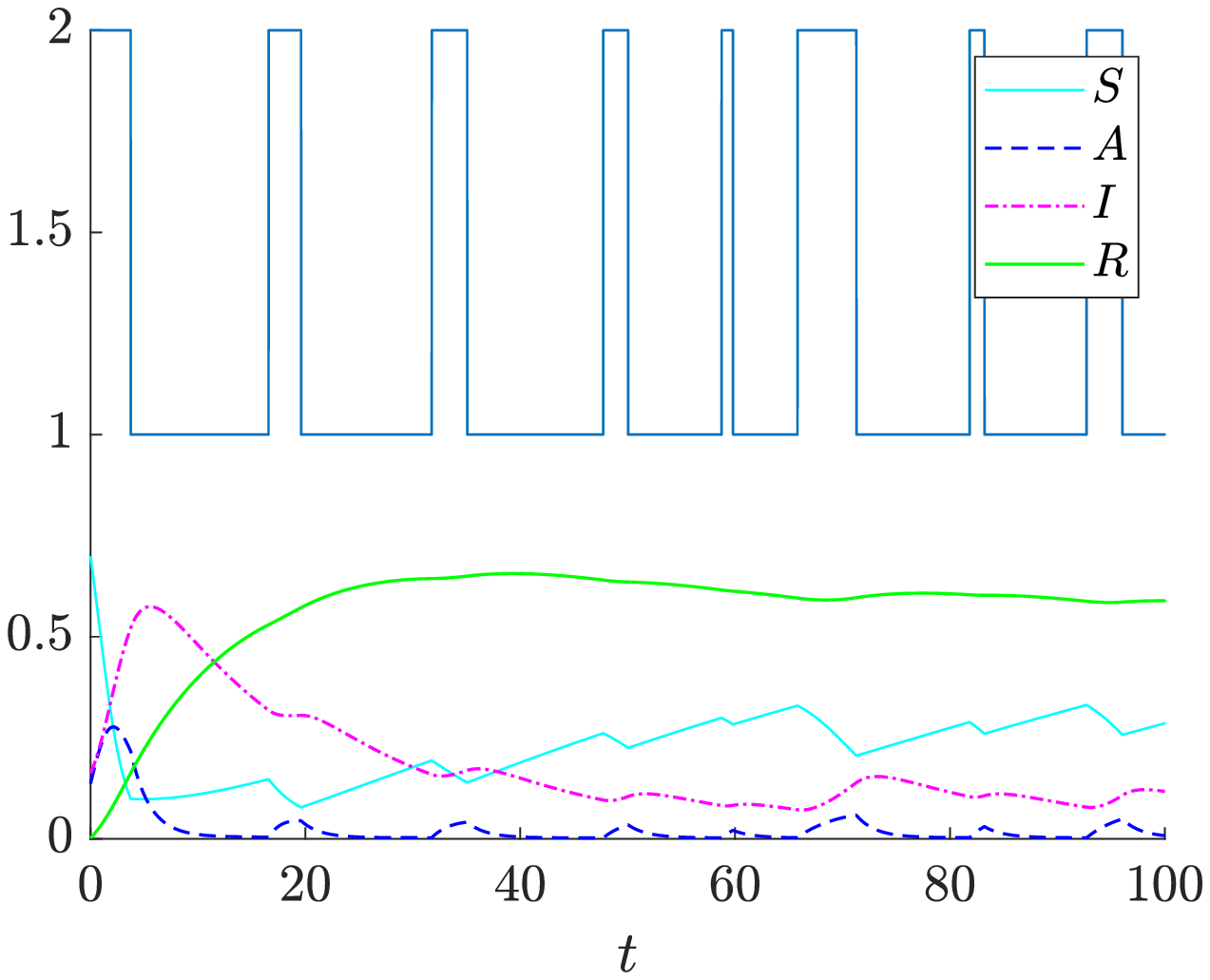}\put(-210,130){b)}
\end{minipage}
 \hskip4mm
    %\subfloat[\label{fig:AQ2}]{%
   \begin{minipage}{0.4\textwidth}
    \includegraphics[width=\textwidth,height=5cm]{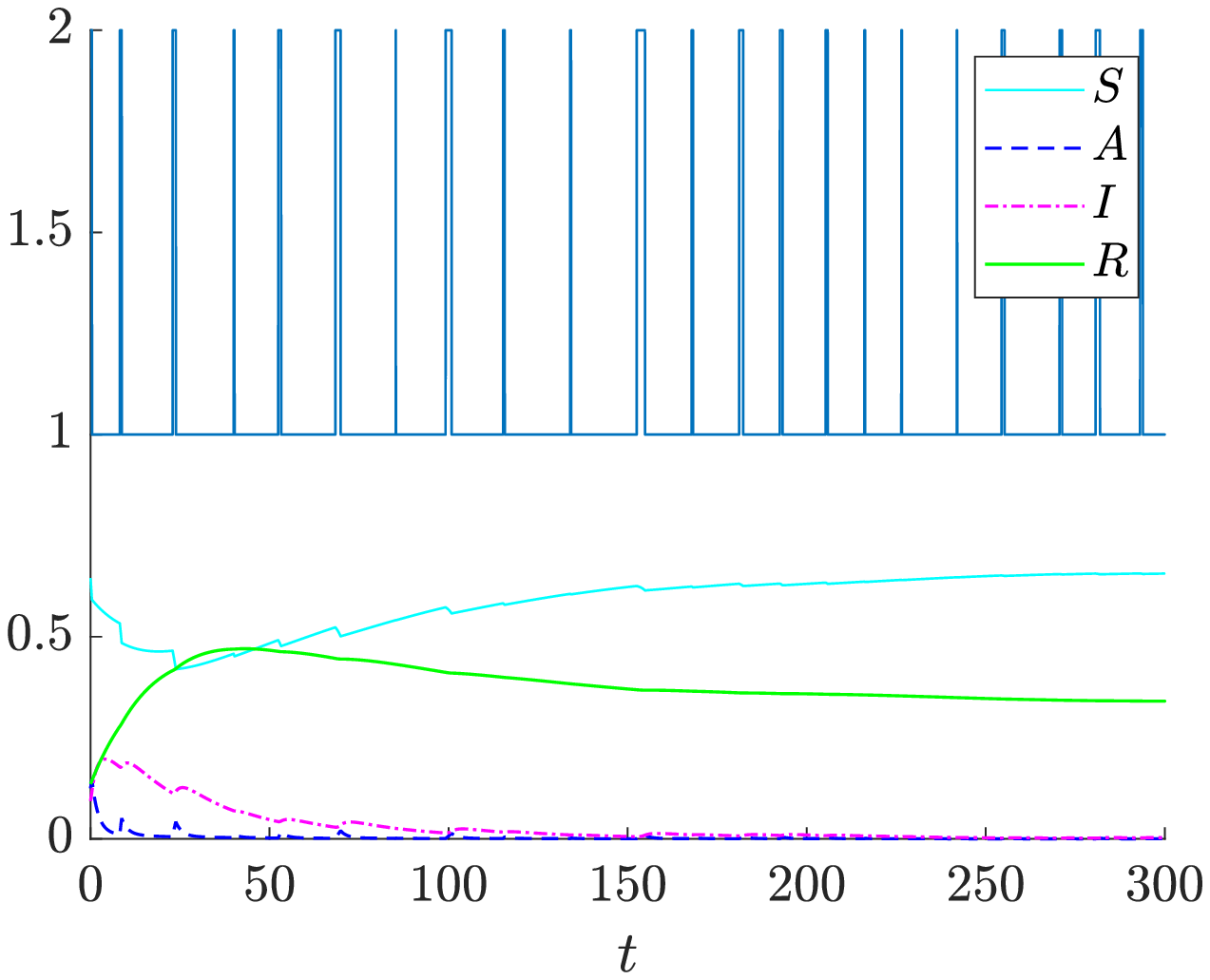}\put(-210,130){c)}
\end{minipage}
    %}
    \caption{Dynamical behaviour of system \eqref{sairs_s} and the semi-Markov process $r(t)$. The parameter values are:  $\beta_A(1)=\beta_I(1)=0.05$, $\beta_A(2)=\beta_I(2)=0.9$, $\delta_A=\delta_I=0.07$, \ste{$\mu=1/(60*365)$}, $\nu=0.01$, $\alpha=0.5$, $\gamma=0.02$. The parameters of $F_1$ and $F_2$ are respectively: a) $k_1=4$, $\theta_1=0.8$ and $k_2=15$, $\theta_2=0.8$, b) $k_1=15$, $\theta_1=0.8$ and $k_2=3$, $\theta_2=0.8$, c)
    $k_1=18$, $\theta_1=0.8$ and $k_2=1$, $\theta_2=0.8$} \label{fig:equal}
  \end{figure}

\emph{Case 3: $\beta_A(r)=\beta_I(r):=\beta(r)$, $\delta_A(r)=\delta_I(r):=\delta(r)$}. We consider two states: $r_1$, in which the epidemic dies out, and $r_2$ in which it persists. The parameters are: $\beta_A(1)=\beta_I(1)=0.05$, $\beta_A(2)=\beta_I(2)=0.9$, $\delta_A=\delta_I=0.07$, \ste{$\mu=1/(60*365)$}, $\nu=0.01$, $\alpha=0.5$, $\gamma=0.02$.
In Fig.\ref{fig:equal} a), we have $F_1$, $F_2$ with parameters $k_1=4$, $\theta_1=0.8$ and $k_2=15$, $\theta_2=0.8$, respectively. Consequently, $m_1=5$ and $m_2=18.75$, and $$\sum_{r \in \cM} \pi_r m_r \left( \beta(r) \frac{\gamma+\mu}{\nu +\gamma + \mu } -(\delta+\mu)\right)=4.8809 >0.$$ Thus, the whole system is persistent in time mean.
In Fig.\ref{fig:equal} b), $F_1$ and $F_2$ have parameters $k_1=15$, $\theta_1=0.8$ and $k_2=3$, $\theta_2=0.8$. Thus, $m_1=18.75$ and $m_2=3.75$, and $$\sum_{r \in \cM} \pi_r m_r \left( \beta(r) \frac{\gamma+\mu}{\nu +\gamma + \mu } -(\delta+\mu)\right)=0.6506 >0.$$ Here, we stay on average longer in the state where the epidemic will go extinct, with respect to the case a). However, although the system is persistent in time mean, the threshold value lowers a lot and the fraction of infectious symptomatic and asymptomatic individuals have the time to decays in some time windows, and the susceptible to increases.
In Fig.\ref{fig:equal} c) $F_1$ and $F_2$ have parameters $k_1=18$, $\theta_1=0.8$ and $k_2=1$, $\theta_2=0.8$, respectively. Thus, $m_1=22.5$ and $m_2=1.25$, and $$\sum_{r \in \cM} \pi_r m_r \left( \beta(r) \frac{\gamma+\mu}{\nu +\gamma + \mu } -(\delta+\mu)\right)=-0.0812 <0.$$ Thus, in this case with the same parameters of the cases a) and b), we have that the disease will go extinct almost surely, in the long run, stressing again the relevance of the mean sojourn times to stem the epidemic.

\section{Conclusion}  
We have analyzed a SAIRS-type epidemic model with vaccination under semi-Markov switching. In this model, the role of the asymptomatic individuals in the epidemic dynamics and the random environment that possibly influences the disease transmission parameters are considered. Under the assumption that both asymptomatic and symptomatic infectious have the same transmission and recovery rates, we have found the value of the basic reproduction number $\cR_0$ for our stochastic model. We have showed that if $\cR_0<1$ the disease will go extinct almost surely, while if $\cR_0>1$ the system is persistent in time mean. Then, we have analyzed the model without restrictions, that is the transmission and recovery rates of the asymptomatic and symptomatic individuals ca be possible different. In this case, we have found two different sufficient conditions for the almost sure extinction of the disease and a sufficient condition for the persistence in time mean. However, the two regions of extinction and persistence are not adjacent but there is a gap between them, thus we do not have a threshold value dividing them. \\
In the case of disease persistence, by restricting the analysis to two environmental states, under the Lie bracket conditions, we have investigated the omega-limit set of the system. Moreover, we have proved the existence of a unique invariant probability measure for the Markov process obtained by introducing the backward recurrence process that keeps track of the time elapsed since the latest switch.
Finally, we have provided numerical simulations to validate our analytical result and investigate the role of the mean sojourn time in the random environments.

\section*{Acknowledgments}    
This research was supported by the University of Trento in the frame ``SBI-COVID - Squashing the business interruption curve while flattening pandemic curve (grant 40900013)''.
   
\bibliographystyle{plain}
\bibliography{biblio2}
\end{document}